\documentclass[12pt]{article}
\usepackage{graphicx}
\usepackage{amsmath}
\usepackage{amssymb}
\usepackage{theorem}

\numberwithin{equation}{section}

\newtheorem{thm}{Theorem}[section]
\newtheorem{lemma}[thm]{Lemma}
\newtheorem{prop}[thm]{Proposition}
\newtheorem{cor}[thm]{Corollary}
{\theorembodyfont{\rmfamily}
\newtheorem{defn}[thm]{Definition}
\newtheorem{eg}[thm]{Example}

\newtheorem{rmk}[thm]{Remark}
}

\newcommand{\qed}{\hfill \mbox{\raggedright \rule{.07in}{.1in}}}

\newenvironment{proof}{\vspace{1ex}\noindent{\bf Proof}\hspace{0.5em}}{\hfill\qed\vspace{1ex}}

\newenvironment{pfof}[1]{\vspace{1ex}\noindent{\bf Proof of #1}\hspace{0.5em}}{\hfill\qed\vspace{1ex}}

\newcommand{\R}{{\mathbb R}}

\newcommand{\Z}{{\mathbb Z}}
\newcommand{\T}{{\mathbb T}}

\newcommand{\dotx}{\dot x^{(\epsilon)}}
\newcommand{\x}{x^{(\epsilon)}}
\newcommand{\z}{z_\epsilon}
\newcommand{\y}{y^{(\epsilon)}}

\newcommand{\hatx}{\hat x^{(\epsilon)}}

\newcommand{\eps}{{\epsilon}}
 \newcommand{\supp}{\operatorname{supp}}
 \newcommand{\diam}{\operatorname{diam}}
 \newcommand{\Lip}{{\operatorname{Lip}}}
 \newcommand{\Leb}{{\operatorname{Leb}}}
\newcommand{\SMALL}{\textstyle}
\newcommand{\BIG}{\displaystyle}

\title{Averaging and Rates of Averaging \\ for Uniform Families \\ of Deterministic Fast-Slow Skew Product Systems}

\author{A. Korepanov \hspace{2em}  Z. Kosloff  \hspace{2em} I. Melbourne \\[.75ex]
 {\small Mathematics Institute, University of Warwick, Coventry, CV4 7AL, UK}}

\date{21 September 2015.   Updated 11 February 2016}

\begin{document}

 \maketitle

\begin{abstract}
We consider families of fast-slow skew product maps of the form
  \begin{align*}
x_{n+1}   = x_n+\eps a(x_n,y_n,\eps), \quad
y_{n+1}   = T_\eps y_n,
  \end{align*}
  where $T_\eps$ is a family of nonuniformly expanding maps, and
prove averaging and rates of averaging for the slow variables~$x$ as $\eps\to0$.   Similar results are obtained also for continuous time systems
  \begin{align*}
\dot x   =  \eps a(x,y,\eps), \quad
\dot y   =  g_\eps(y).
  \end{align*}

Our results include cases where the family of fast dynamical systems consists of intermittent maps, unimodal maps (along the Collet-Eckmann parameters) and Viana maps.
\end{abstract}

\section{Introduction}

The classical Krylov-Bogolyubov averaging method~\cite{Sanders2007} deals with skew product flows of the form
  \begin{align*}
\dot x   =  \eps a(x,y,\eps), \quad
\dot y   =  g(y).
  \end{align*}
Let $\nu$ be an ergodic invariant probability measure for the fast flow generated by $g$.  Under a uniform Lipschitz condition on $a$, it can be shown that solutions to the slow $x$ dynamics, suitably rescaled, converge almost surely to solutions of an averaged ODE $\dot X=\bar a(X)$ where $\bar a(x)=\int a(x,y,0)\,d\nu(y)$.    

A considerably harder problem is to handle the fully-coupled situation
  \begin{align*}
\dot x   =  \eps a(x,y,\eps), \quad
\dot y   =  g(x,y,\eps).
  \end{align*}
Here it is supposed that there is a distinguished family of ergodic invariant probability measures $\nu_{x,\eps}$ for the fast vector fields $g(x,\cdot,\eps)$ and the averaged vector field is given by $\bar a(x)=\int a(x,y,0)\,d\nu_{x,0}(y)$.
The first results on averaging for fully-coupled systems were due to
Anosov~\cite{Anosov60} who considered the case where the fast vector fields are Anosov with $\nu_{x,\eps}$ absolutely continuous.  Convergence here is in the sense of convergence in probability with respect to Lebesgue measure.

Kifer~\cite{Kifer04,Kifer05} extended the results of~\cite{Anosov60} to the case where the fast vector fields are Axiom~A (uniformly hyperbolic) with SRB measures $\nu_\eps$.
More generally, Kifer considers the case where
$x\mapsto\nu_{x,0}$ is sufficiently regular so that $\bar a$ is Lipschitz,
and gives necessary and sufficient conditions for averaging to hold.
However, the only situations where the conditions in~\cite{Kifer04,Kifer05} are verified are in the Axiom~A case, even though it is hoped~\cite{Kifer05} that the conditions are verifiable for nonuniformly hyperbolic examples.
Analogous results for the discrete time case are obtained in~\cite{Kifer03}.
See also~\cite[Theorem~5]{Dolgopyat05} for certain partially hyperbolic fast vector fields.

Here, we consider an intermediate class of examples that lies between the classical uncoupled situation and the fully coupled systems of~\cite{Anosov60, Kifer04}, namely families of skew products of the form
  \begin{align} \label{eq-intro}
\dot x   =  \eps a(x,y,\eps), \quad
\dot y   =  g(y,\eps),
  \end{align}
with distinguished family of ergodic invariant measures $\nu_\eps$ and averaged vector field $\bar a(x)=\int a(x,y,0)\,d\nu_0(y)$.
Notice that in this way we avoid issues concerned with the regularity of the averaged vector field $\bar a$, but we still have to deal with the $\eps$-dependence of the measures $\nu_\eps$ as well as the fast vector fields.
In other words, linear response (differentiability) of the invariant measures is replaced by statistical stability (weak convergence) which is more tractable.  Indeed one aspect of the general framework in this paper is that our averaging theorems hold in a similar generality to the methods of Alves \& Viana~\cite{Alves04,AlvesViana02} for proving statistical stability.

Hence, we obtain results on averaging and rates of averaging for a large class of families of skew products~\eqref{eq-intro}, going far beyond the uniformly hyperbolic setting, both in discrete and continuous time.  Our examples include situations where the fast dynamics is given by intermittent maps with arbitrarily poor mixing properties, unimodal maps where linear response fails, and flows built as suspensions over such maps.

We obtain results also on rates of averaging.  In the very simple situation
$\dot x=\eps a(y)$, $\dot y=g(y)$, where $g$ is a uniformly expanding semiflow or uniformly hyperbolic flow, it is easily seen that the optimal rate of averaging in $L^1$ is $O(\eps^{1/2})$.  For systems of the form~\eqref{eq-intro}, we often obtain the essentially optimal rate $O(\eps^{\frac12-})$.\footnote{$q-$ denotes $q-a$ for all $a>0$.}

We have chosen to focus in this paper on the case of noninvertible dynamical systems.  In this situation, the measures of interest are absolutely continuous and we are able to present the main ideas without going into the technical issues presented by dealing with  nonabsolutely continuous measures as required in the invertible setting.   The invertible case will be covered in a separate paper.

Even in the noninvertible setting, our results depend strongly on extensions and clarifications of the classical second order averaging theorem.  These prerequisites are presented in Appendix~\ref{sec-second} and may be of independent interest.

\vspace{1ex}
The remainder of the paper is organised as followed.
In Section~\ref{sec-gen}, we set up the averaging problem for families of fast-slow skew product systems in the discrete time case, leading to a general result Theorem~\ref{thm-gen} for such systems.   In Section~\ref{sec-UE}, we show that Theorem~\ref{thm-gen} leads easily to averaging when the fast dynamics is a family of uniformly expanding maps.
Section~\ref{sec-ufNUE} is the heart of the paper and deals with the case when the fast dynamics is a family of nonuniformly expanding maps.
Our main examples are presented in Section~\ref{sec-egNUE}.
In Section~\ref{sec-flow}, we show how the continuous time case reduces to the discrete time case.   In Section~\ref{sec-ae}, we present a simple example to show that almost sure convergence fails for families of skew products.

\section{General averaging theorem for families of skew products}
\label{sec-gen}

Let $T_\eps:M\to M$, $0\le\eps<\eps_0$, be a family of transformations
defined on a measurable space $M$.
For each $\eps\in[0,\eps_0)$, let $\nu_\eps$ denote 
a $T_\eps$-invariant ergodic probability measure on $M$.

We consider the family of fast-slow systems
  \begin{align} \label{eq-fastslow} \nonumber
\x_{n+1}  & = \x_n+\eps a(\x_n,\y_n,\eps), \quad \x_0=x_0,   \\
\y_{n+1} &  = T_\eps\y_n, \quad \y_0=y_0, 
  \end{align}
where the initial condition $\x_0=x_0$ is fixed throughout.
The initial condition $y_0\in M$ is chosen randomly with respect to various measures that are specified in the statements of the results.
Here $a:\R^d \times M\times[0,\eps_0)\to \R^d$
is a family of functions satisfying certain regularity hypotheses.

  Define $\bar a(x)=\int_M a(x,y,0)\,d\nu_0(y)$ and consider
the ODE
\begin{align} \label{eq-ODE}
\dot X= \bar a(X),\enspace X(0)=x_0.
\end{align}
We are interested in the convergence, and rate of convergence, of the slow variables $\x_n$, suitably rescaled, to solutions $X(t)$ of this ODE.
More precisely, 
  define $\hatx:[0,1]\to\R^d$ by setting $\hatx(t)=\x_{[t/\eps]}$.
We study convergence of the difference 
\[
\z=\sup_{t\in[0,1]}|\hatx(t)-X(t)|.
\]

\begin{rmk}
The restriction to the time interval $[0,1]$ entails no loss of generality: the results apply to arbitrary bounded intervals by rescaling $\eps$.
\end{rmk}

\paragraph{Regularity assumptions}
Given a function $g:\R^d\to\R^n$, we define
$\|g\|_\Lip=\max\{|g|_\infty,\Lip\, g\}$ where
$\Lip\,g=\sup_{x\neq x'}|g(x)-g(x')|/|x-x'|$ 
and $|x-x'|=\max_{i=1,\dots,n}|x_i-x'_i|$.

In this section, and also in Appendix~\ref{sec-second}, we consider functions $g:\R^d\times M\times[0,\eps_0)\to\R^n$
	where there is no metric structure assumed on $M$.
In that case, $\|g\|_\Lip=\sup_{y\in M}\sup_{\eps\in[0,\eps_0)}\|g(\cdot,y,\eps)\|_\Lip$.
If $E\subset \R^d$, then 
$\|g|_E\|_\Lip$ is computed by restricting to $x,x'\in E$ (and $y\in M$, $\eps\in[0,\eps_0)$).

Throughout, we write $D=\frac{\partial}{\partial x}$.
If $g:\R^d\times M\times[0,\eps_0)\to\R^n$, then
$Dg:\R^d\times M\times[0,\eps_0)\to \R^{n\times d}$ and 
$\|Dg\|_\Lip$ is defined accordingly.
Similarly for $\|Dg|_E\|_\Lip$ when $E\subset\R^d$.

Below, $L_1,L_2,L_3\ge1$ are constants.
We require that $a$ is globally Lipschitz in $x$:
\begin{align} \label{eq-L1}
\| a\|_\Lip \leq L_1.
\end{align}
Set $E=\{x\in\R^d:|x-x_0|\le L_1\}$.  
We assume that $Da|_E$ is Lipschitz in $x$:
\begin{align} \label{eq-L2}
\| Da|_E\|_\Lip \leq L_2.
\end{align}
and that
\begin{align} \label{eq-L3}
\sup_{x\in E}\sup_{y\in M}|a(x,y,\eps)-a(x,y,0)|\le L_3\eps.
\end{align}
In the sequel we let
$L=\max\{L_1,L_2,L_3\}$.

\subsection{Order functions and a general averaging theorem}
\label{sec-order}

  Define $\bar a(x,\eps)=\int_M a(x,y,\eps)\,d\nu_\eps(y)$ and let
$v_{\eps,x}(y)=a(x,y,\eps)-\bar a(x,\eps)$.
Set
\begin{align*}
  \delta_{1,\eps} & = \sup_{x\in E} \sup_{1 \leq n \leq 1/\eps} 
  \eps|v_{\eps, x, n}|, \quad \text{where} \quad 
v_{\eps, x, n}= \sum_{j=0}^{n-1} v_{\eps,x}\circ T_\eps^j, \\
\delta_{2,\eps} & = \sup_{x\in E}\sup_{1\leq n\leq 1/\epsilon}\epsilon|V_{\eps,x,n}|,
\quad\text{where} \quad
V_{\eps, x, n}= \sum_{j=0}^{n-1} (Dv_{\eps,x})\circ T_\eps^j.
\end{align*}
Then we define the \emph{order function} $\delta_\eps=\delta_{1,\eps}+\delta_{2,\eps}:M\to\R$.

Finally, define
$S_\eps  = \sup_{x\in E}|\int_M a(x,y,0)\,(d\nu_\eps-d\nu_0)(y)|+\eps$.

\begin{thm} \label{thm-gen}  
If $\delta_\eps\le\frac12$,
then $\z  \le 6e^{2L}(\delta_\eps+S_\eps)$.
\end{thm}

The proof of Theorem~\ref{thm-gen} is postponed to Appendix~\ref{sec-second}.

\begin{rmk}  
For averaging without rates, it suffices instead of condition~\eqref{eq-L3} that
$\lim_{\eps\to0}|a(x,y,\eps)-a(x,y,0)|=0$ for all $x\in E$, $y\in M$.
\end{rmk}

\begin{rmk} \label{rmk-ae}
	As shown in Section~\ref{sec-ae}, almost sure convergence in the averaging theorem is not likely to hold for fast-slow systems of type~\eqref{eq-fastslow}.  Hence we consider convergence in $L^q$ with respect to certain absolutely continuous probability measures on $M$.  Since $\z\le 2L$ and $\delta_\eps\le 4L$, convergence in $L^p$ is equivalent to convergence in $L^q$ for all $p,q\in(0,\infty)$.  For brevity, we restrict statements to convergence in $L^1$ except when speaking of rates.
\end{rmk}

\begin{rmk} \label{rmk-L2}
If condition~\eqref{eq-L2} fails, then all of our results without rates go through unchanged.  Moreover, it is still possible to obtain results with rates but usually with weaker rates of convergence (the best rates are $O(\eps^{\frac14-})$ instead of $O(\eps^{\frac12-})$).  These results are obtained by using $\delta_{1,\eps}$ (first order averaging) instead of $\delta_\eps=\delta_{1,\eps}+\delta_{2,\eps}$ (second order averaging) and can be found in the first (much longer) version of this paper~\cite{KKMsub1}.
\end{rmk}

According to Theorem~\ref{thm-gen}, results on averaging reduce to estimating the scalar quantity $S_\eps$ and the random variable $\delta_\eps=\delta_\eps(y_0)$.  These quantities are discussed below in Subsections~\ref{sec-ss} and~\ref{sec-delta} respectively.

\subsection{Statistical stability}
\label{sec-ss}
In this subsection, we suppose that $M$ is a topological space and that the $\sigma$-algebra of measurable sets is the $\sigma$-algebra of Borel sets.
Recall that the family of measures $\nu_\eps$ is {\em statistically stable} at $\eps=0$ if $\nu_0$ is the weak limit of $\nu_\eps$ as $\eps\to0$
($\nu_\eps\to_w\nu_0$).  This means that $\int_M \phi\,d\nu_\eps\to
\int_M \phi\,d\nu_0$ for all continuous bounded functions $\phi:M\to\R$.

In the noninvertible setting, often a stronger property known as strong statistical stability holds.  Let $m$ be a reference measure on $M$ and suppose that
$\nu_\eps$ is absolutely continuous with respect to $m$ for all $\eps\ge0$.
Then $\nu_0$ is {\em strongly statistically stable} if the densities
$\rho_\eps=d\nu_\eps/dm$ satisfy $\lim_{\eps\to0}R_\eps=0$ where
$R_\eps=\int_M|\rho_\eps-\rho_0|\,dm$.
We note that $S_\eps\le LR_\eps+\eps$.

\begin{prop} \label{prop-ss}
If $\nu_\eps\to_w\nu_0$, then $\lim_{\eps\to0}S_\eps=0$.
\end{prop}

\begin{proof}
Let $A_\eps(x)=\int_M a(x,y,0)\,d\nu_\eps(y)-\int_M a(x,y,0)\,d\nu_0(y)$.
Let $\delta>0$.
Since $\nu_\eps\to_w\nu_0$, we have that $A_\eps(x)\to0$ for each $x$, so
there exists $\eps_x>0$
such that $|A_\eps(x)|<\delta$ for all $\eps\in(0,\eps_x)$.
Moreover, $|A_\eps(z)|<2\delta$ for all $\eps\in(0,\eps_x)$ and
$z\in B_{\delta/(2L)}(x)$.  Since $E$ is covered by finitely many such balls
$B_{\delta/(2L)}(x)$, there exists $\bar\eps>0$ such that
$\sup_{x\in E}|A_\eps(x)|<2\delta$ for all $\eps\in(0,\bar\eps)$.
 Hence $\int_M a(x,y,0)\,(d\nu_\eps-d\nu_0)(y)$ converges to zero uniformly in $x$.
\end{proof}

Hence for proving averaging theorems, statistical stability takes care of the term $S_\eps$ in Theorem~\ref{thm-gen}.
In specific examples, we are able to appeal to results on statistical stability with rates, yielding effective estimates.  

\begin{prop} \label{prop-gen}
Let $q\geq1$.  There is a constant $C>0$ such that
\begin{align*}
|\z|_{L^q(\nu_\eps)} & \le C(|\delta_\eps|_{L^q(\nu_\eps)}+S_\eps), 
\end{align*}
for all $\eps\in[0,\eps_0)$.

If the measures $\nu_\eps$ are absolutely continuous with respect to $m$, then
there is a constant $C>0$ such that
\begin{align*}
|\z|_{L^q(\nu_0)} & \le C(|\delta_\eps|_{L^q(\nu_\eps)}+R_\eps^{1/q}+\eps),
\end{align*}
for all $\eps\in[0,\eps_0)$.
\end{prop}

\begin{proof}
Let $A_\eps=\{y\in M: \delta_\eps(y)\le\frac12\}$.
Then Theorem~\ref{thm-gen} applies on $A_\eps$ and
\begin{align*}
 \int_M\z^q\,d\nu_\eps & =
\int_{M\setminus A_\eps} \!\!\! \z^q\,d\nu_\eps+
\int_{A_\eps}\z^q\,d\nu_\eps 
 \le (2L)^q\nu_\eps(\delta_\eps>{\SMALL\frac12})+(6e^{2L})^q\int_M(\delta_\eps+S_\eps)^q\, d\nu_\eps \\
& \le  (4L)^q\int_M \delta_\eps^q\,d\nu_\eps
+(6e^{2L})^q\int_M (\delta_\eps+S_\eps)^q\,d\nu_\eps
\le (12e^{2L})^q\int_M (\delta_\eps+S_\eps)^q\,d\nu_\eps.
\end{align*}
Hence 
\[
|\z|_{L^q(\nu_\eps)}\le 12e^{2L}|\delta_\eps+S_\eps|_{L^q(\nu_\eps)}
\le 12e^{2L}|\delta_\eps|_{L^q(\nu_\eps)}+12e^{2L}S_\eps,
\]
yielding the first estimate.

Next,
\begin{align*}
\int_M\z^q\,d\nu_0 & = \int_M\z^q\,d\nu_\eps+ \int_M\z^q\,(d\nu_0-d\nu_\eps)
\le (12e^{2L})^q\int_M (\delta_\eps+S_\eps)^q\,d\nu_\eps+(2L)^qR_\eps \\
& \le (12e^{2L})^q\int_M (\delta_\eps+LR_\eps+\eps)^q\,d\nu_\eps+(2L)^qR_\eps \\
& \le (12Le^{2L})^q \int_M (\delta_\eps+R_\eps^{1/q}+\eps)^q\,d\nu_\eps+(2L)^qR_\eps  \\
& \le (24Le^{2L})^q \int_M (\delta_\eps+R_\eps^{1/q}+\eps)^q\,d\nu_\eps.
\end{align*}
Hence
\[
|\z|_{L^q(\nu_0)}\le 24Le^{2L} |\delta_\eps+R_\eps^{1/q}+\eps|_{L^q(\nu_\eps)}
\le 24Le^{2L}(|\delta_\eps|_{L^q(\nu_\eps)}+R_\eps^{1/q}+\eps),
\]
yielding the second estimate.
\end{proof}

\begin{cor} \label{cor-gen}
(a)  Assume statistical stability and that $\lim_{\eps\to0}\int_M \delta_\eps\,d\nu_\eps=0$.
Then $\lim_{\eps\to0}\int_M\z\,d\nu_\eps=0$.

\vspace{1ex}\noindent(b)
Assume in addition strong statistical stability and that $\mu$ is a probability measure on $M$ with $\mu\ll \nu_0$.
Then $\lim_{\eps\to0}\int_M\z\,d\mu=0$.
\end{cor}

\begin{proof}
Part (a), and part~(b) in the special case $\mu=\nu_0$, are immediate from Proposition~\ref{prop-gen}.
To prove the general case of part~(b),
suppose for contradiction that $\int_M z_{\eps_k}\,d\mu\to b>0$ along some subsequence $\eps_k\to0$.  
Since 
$\int_M z_{\eps_k}\,d\nu_0\to 0$, by passing
without loss to a further subsequence, we can suppose also that 
$ z_{\eps_k}\to 0$ on a set of full measure with respect to $\nu_0$ and hence
with respect to $\mu$.
By the bounded convergence theorem, 
$\int_M z_{\eps_k}\,d\mu\to 0$ which is the desired contradiction.
\end{proof}

The next result is useful in situations where $\nu_0$ is absolutely continuous but whose support is not the whole of $M$.

\begin{cor} \label{cor-dm}
Assume strong statistical stability and that
$\lim_{\eps\to0}\int_M \delta_\eps\,d\nu_0=0$.
  Suppose further that each $T_\eps$ is nonsingular with respect to $m$ and that
for almost every $y\in M$, there exists $N\ge1$ such that
$T_\eps^Ny\in\supp\nu_0$ for all $\eps\in[0,\eps_0)$.    Then 
$\lim_{\eps\to0}\int_M\z\,d\mu=0$
for every probability measure $\mu$ on $M$ with $\mu\ll m$.
\end{cor}

\begin{proof}
First, we note that for all $N\ge1$, $\eps\ge0$,
\begin{align} \label{eq-dm}
|\delta_\eps\circ T_\eps^N-\delta_\eps|_\infty\le 8LN\eps .
\end{align}

By the arguments in the proof of Corollary~\ref{cor-gen},
it suffices to prove that
\mbox{$\int_M\delta_\eps\,dm\to0$}.
Suppose this is not the case.
By Corollary~\ref{cor-gen}(b), \mbox{$\int_{\supp\nu_0}\delta_\eps\,dm\to0$}.
Hence there exists a subsequence $\eps_k\to0$ and a subset $A\subset\supp\nu_0$ with $m(\supp\nu_0\setminus A)=0$ such that
(i) $\delta_{\eps_k}\to0$ pointwise on $A$ and (ii) $\int_M\delta_{\eps_k}\,dm\to b>0$.

Since each $T_\eps$ is nonsingular, there exists $M'\subset M$ with $m(M')=1$
such that $M'\cap T_{\eps_k}^{-n}(\supp\nu_0\setminus A)=\emptyset$ for
all $k,n\ge1$.
By the hypothesis of the result, there is a subset $M''\subset M'$ with $m(M'')=1$ such that
for any $y\in M''$ there exists $N\ge1$ such that
$T_{\eps_k}^Ny\in A$ for all $k\ge1$.
Hence it follows from (i) and~\eqref{eq-dm} that $\delta_{\eps_k}\to 0$ pointwise on $M''$.
By the bounded convergence theorem $\int_M\delta_{\eps_k}\,dm\to0$.  Together with (ii), this yields the desired contradiction.
\end{proof}

\begin{rmk} \label{rmk-dm}
The hypotheses of Corollary~\ref{cor-dm} are particularly straightforward to apply when $\supp\nu_0$ has nonempty interior.
This property is automatic for large classes of nonuniformly expanding maps, see~\cite[Lemma~5.6]{AlvesBonattiViana00} (taking $G=\supp\nu_0\cap H(\sigma)$, it follows that $\bar G$ contains a disk).

\end{rmk}

\subsection{Estimating the order function}
\label{sec-delta}

By Proposition~\ref{prop-gen} and Corollary~\ref{cor-gen}, it remains to deal with the
order function $\delta_\eps$.  This is a random variable depending on the initial condition $y_0$.   Here we give a useful estimate.

Since $\delta_\eps=\delta_{1,\eps}+\delta_{2,\eps}$ and the definition of $\delta_{2,\eps}$ is identical to that of $\delta_{1,\eps}$ with $a$ replaced by $Da$,
it suffices to consider $\delta_{1,\eps}$.
From now on, $\Gamma$ denotes a constant that only depends on $d,p,L$ and whose value may change from line to line.

\begin{lemma}\label{lem-delta}
  Let $\mu$ be a probability measure on $M$.
  Then for all $p \geq 0$, $\eps\in[0,\eps_0)$,
  \begin{align*}
    \int_M \delta_{1,\eps}^{p+d+1} \, d\mu \leq 
    \Gamma \eps^p\sup_{x\in E} \sup_{1 \leq n \leq 1/\eps} \int_M |v_{\eps, x, n}|^p \, d\mu.
  \end{align*}
\end{lemma}

\begin{proof}
	For most of the proof, we work pointwise on $M$ suppressing the initial condition $y_0\in M$.
        There exist $\tilde x\in E$, $\tilde n\in[0,1/\eps]$ such that
  $\delta_{1,\eps}=\epsilon|v_{\eps, \tilde{x}, \tilde{n}}|$.
  Observe that
  \[
          |v_{\eps,x,n}-v_{\eps, \tilde{x}, \tilde{n}}| \leq
    |v_{\eps, x, n} - v_{\eps, \tilde{x}, n}| +
    |v_{\eps, \tilde{x}, n} - v_{\eps, \tilde x, \tilde n}| \leq 2L\eps^{-1}|x-\tilde{x}| + 2L |n-\tilde{n}|
  \]
  for every $x\in E$ and $n\le 1/\eps$. Define
  \begin{align*}
          A &= \{x  \in E : |x-\tilde{x}| \leq {\SMALL\frac{1}{8L}}\delta_{1,\eps} \},
          \quad 
          B = \{ n \in [0, \eps^{-1}] : \eps |n-\tilde{n}| \leq {\SMALL\frac{1}{8L}}\delta_{1,\eps}\}.
  \end{align*}
  Then for every $x \in A$ and $n \in B$
  we have $\eps|v_{\eps, x, n}| \geq \delta_{1,\eps}/2$.
  Moreover, since $\delta_{1,\eps}\le2L$, we have $\Leb(A)\ge\Gamma\delta_{1,\eps}^d$
  and $\# B\ge \eps^{-1} \delta_{1,\eps} / 8L$.
  Hence
  \[
    \eps^p\sum_{n=0}^{[1/\eps]-1} \int_E |v_{\eps, x, n}|^p \, dx  \geq 
    (\# B)\, \Leb(A)\left(\frac{\delta_{1,\eps}}{2}\right)^p\ge 
    \Gamma\eps^{-1}\delta_{1,\eps}^{p+d+1}.
  \]

  Finally,
  \begin{align*}
    \int_M \delta_{1,\eps}^{p+d+1}\,d\mu
    & \le \Gamma
      \eps^{p+1} \sum_{n=0}^{[1/\eps]-1} \int_E\int_M |v_{\eps, x, n}|^p \,
      d\mu \,dx
  \le \Gamma \eps^p\sup_{x\in E} \, \sup_{1 \leq n \leq 1/\eps}
      \int_M |v_{\eps, x, n}|^p \,d\mu,
  \end{align*}
  as required.
\end{proof}

\begin{rmk}
Often, estimating  $\int_M|v_{\eps,x,n}|\,d\mu$ leads to an essentially identical estimate for $\int_M \sup_{1\le n\le 1/\eps}|v_{\eps,x,n}|\,d\mu$.
In this case, slightly better convergence rates for $\delta_{1,\eps}$ can be obtained using the estimate
  \begin{align} \label{eq-delta}
    \int_M \delta_{1,\eps}^{p+d} \, d\mu \leq 
    \Gamma \eps^p\sup_{x\in E} \int_M \sup_{1\le n\le 1/\eps}|v_{\eps, x, n}|^p \, d\mu
  \end{align}
  for all $p \geq 0$, $\eps\in[0,\eps_0)$.
\end{rmk}

\section{Examples: Uniformly expanding maps}
\label{sec-UE}

Let $T_\eps:M\to M$ be a family of maps defined on a metric space $(M,d_M)$ with invariant ergodic Borel probability measures $\nu_\eps$.   Let $P_\eps$ denote the corresponding transfer operators,
so $\int_M P_\eps v\,w\,d\nu_\eps=\int_M v\,w\circ T_\eps\,d\nu_\eps$ for all
$v\in L^1(\nu_\eps)$, 
$w\in L^\infty(\nu_\eps)$. 

From now on, we require Lipschitz regularity in the $M$ variables in addition to the $\R^d$ variables as was required in assumptions~\eqref{eq-L1} and~\eqref{eq-L2}.
So for $g:\R^d\times M\times[0,\eps_0)\to\R^n$ we define
	$\|g\|_\Lip=|g|_\infty+\sup_{\eps\in[0,\eps_0)}\Lip\,g(\cdot,\cdot,\eps)$ where
$\Lip\,g(\cdot,\cdot,\eps)=\sup_{x\neq x'}\sup_{y\neq y'}|g(x,y,\eps)-g(x',y',\eps)|/(|x-x'|+d_M(y,y'))$.
We continue to assume conditions~\eqref{eq-L1}---\eqref{eq-L3} with this modified definition of $\|\;\|_\Lip$.

\begin{prop} \label{prop-expanding}
Suppose that there is a sequence of constants $a_n\to0$ such that
$\int_M|P_\eps^nv-\int_M v\,d\mu_\eps|\le a_n\|v\|_\Lip$ for all Lipschitz $v:M\to\R$ and all $n\ge1$, $\eps\ge0$.  Then $\lim_{\eps\to0}\int_M\delta_\eps\,d\nu_\eps=0$.
\end{prop}

\begin{proof}  
We prove the result for $\delta_{1,\eps}$ and $\delta_{2,\eps}$ separately.
By Remark~\ref{rmk-ae}, we can work in $L^q$ for any choice of $q$ and we take $q=d+3$.

For every Lipschitz $v$, we have
\begin{align*}
\int_M \Bigl(\sum_{j=0}^{n-1}v\circ T_\eps^j\Bigr)^2\,d\nu_\eps & =
\sum_{j=0}^n\int_M v^2\circ T_\eps^j\,d\nu_\eps+2\sum_{0\le i<j\le n-1} \int_M v\circ T_\eps^i \,v\circ T_\eps^j\,d\nu_\eps
\\ & = n\int_M v^2\,d\nu_\eps+2\sum_{1\le k\le n-1} (n-k) \int_M v\,v\circ T_\eps^k\,d\nu_\eps
\\ & =  n\int_M v^2\,d\nu_\eps+2\sum_{1\le k\le n-1}(n-k) \int_M P_\eps^k v\,v\,d\nu_\eps.
\end{align*}
Hence for $v$ Lipschitz and mean zero,
\[
\int_M \Bigl(\sum_{j=0}^{n-1}v\circ T_\eps^j\Bigr)^2\,d\nu_\eps  
 \le b_n\|v\|_\Lip^2,
\]
where $b_n= n+2n\sum_{1\le k\le n-1} a_k=o(n^2)$ by the assumption on $a_n$.

By condition~\eqref{eq-L1}, $\|v_{\eps,x}\|_\Lip\le 2L$ for all $\eps,x$, so
$\sup_{x\in E}\sup_{1\le n\le 1/\eps}\int_M|v_{\eps,x,n}|^2\,d\nu_\eps\le 4L^2b_n$.
By Lemma~\ref{lem-delta}, it follows that
$\lim_{\eps\to0}\int_M\delta_{1,\eps}^{d+3}\,d\nu_\eps=0$.
Similarly, 
by condition~\eqref{eq-L2}, $\|V_{\eps,x}\|_\Lip\le 2L$, so
$\sup_{x\in E}\sup_{1\le n\le 1/\eps}\int_M|V_{\eps,x,n}|^2\,d\nu_\eps\le 4L^2b_n$ and hence
$\lim_{\eps\to0}\int_M\delta_{2,\eps}^{d+3}\,d\nu_\eps=0$.
\end{proof}

\begin{rmk}  The proof uses only that $n^{-1}\sum_{k=1}^n a_n\to0$.
\end{rmk}

Proposition~\ref{prop-expanding} is useful in situations where $T_\eps$ is a family of (piecewise) uniformly expanding maps.
A general result of Keller \& Liverani~\cite{KellerLiverani99} guarantees uniform spectral properties of the transfer operators $P_\eps$ under mild conditions, and consequently 
$\lim_{\eps\to0}\int_M\delta_\eps^q\,d\nu_\eps=0$ for all $q$.
We mention two situations where this idea can be applied.
Again, for brevity we work in $L^1$ except when discussing convergence rates
(see Remark~\ref{rmk-ae}).

\begin{eg}[Uniformly expanding maps] \label{eg-UE}
Suppose that $M=\T^k\cong\R^/\Z^k$ is a torus with Haar measure $m$
and distance $d_M$ inherited from Euclidean distance on $\R^k$ and normalised so that $\diam M=1$.  
We say that a $C^2$ map $T:M\to M$ is {\em uniformly expanding} if
there exists $\lambda>1$ such that
$|(DT)_yv|\ge\lambda|v|$ for all $y\in M$, $v\in \R^k$.
There is a unique absolutely continuous invariant
probability measure, and the density is $C^1$ and nonvanishing.

If $T_\eps:M\to M$, $\eps\in[0,1]$, is a continuous family of $C^2$ maps,
each of which is uniformly expanding, with corresponding probability measures $\nu_\eps$, then it follows 
from~\cite{KellerLiverani99} that we are in the situation of Proposition~\ref{prop-expanding}, and so 
$\lim_{\eps\to0}\int_M\delta_\eps\,d\nu_\eps=0$.

Moreover, it is well-known that $\nu_0$ is uniformly equivalent to $m$ and is strongly statistically stable.
Hence by Corollary~\ref{cor-gen} we obtain the averaging result
$\lim_{\eps\to0}\int_M \z\,d\nu_\eps=0$ and
$\lim_{\eps\to0}\int_M \z\,d\mu=0$
for every absolutely continuous probability measure $\mu$.

Suppose further that $T_\eps:M\to M$, $\eps\in[0,1]$, is a $C^k$ family of $C^2$ maps, for some $k\in(0,1]$.
By standard results (for instance~\cite{KellerLiverani99} with Banach spaces $C^0$ and $C^1$),
$R_\eps=\int_M|\rho_\eps-\rho_0|\,dm=O(\eps^k)$.
If $k\in(0,\frac12)$, then
by Remark~\ref{rmk-rate} below we obtain the convergence rate $O(\eps^k)$
for $\z$ in $L^q(\nu_\eps)$ and $L^q(m)$ for all $q>0$.
If $k\ge \frac12$, then the convergence rate for $\z$ is $O(\eps^{\frac12-})$.
\end{eg}

\begin{eg}[Piecewise uniformly expanding maps]
Let $M=[-1,1]$ with Lebesgue measure $m$.  We consider continuous maps $T:M\to M$ with $T(-1)=T(1)=-1$
such that $T$ is $C^2$ on $[-1,0]$ and $[0,1]$.  We require that there exists $\lambda>1$ such that
$T'\ge \lambda$ on $[-1,0)$ and
$T'\le -\lambda$ on $(0,1]$.
There exists a unique absolutely continuous invariant probability measure
with density of bounded variation.

The results are analogous to those in Example~\ref{eg-UE}.
Let $\eps\to T_\eps$ be a continuous family of such maps on $[-1,1]$ with associated measures $\nu_\eps$.
We assume that 
$T_0$ is topologically mixing on the interval $[T_0^2(0),T_0(0)]$ and that $0$ is not periodic (which guarantees that $T_\eps$ is mixing for $\eps$ small).

 Then $\nu_0$ is strongly statistically stable, so by Corollaries~\ref{cor-gen} and~\ref{cor-dm}
we obtain averaging in $L^1(\nu_\eps)$ and also in $L^1(\mu)$ for every absolutely continuous probability measure $\mu$.

Now suppose that $\eps\to T_\eps$ is a $C^1$ family of such maps on $[-1,1]$
with densities $\rho_\eps=d\nu_\eps/dm$.
Keller~\cite{Keller82} showed that $\eps\to \rho_\eps$ is $C^{1-}$ as a map into $L^1$ densities.
Hence we obtain the convergence rate $O(\eps^{\frac12-})$ for $\z$ in $L^q(\nu_\eps)$, for all $q>0$ and in $L^1(\nu_0)$.

More precisely,~\cite{Keller82} shows that
$\int_M|\rho_\eps-\rho_0|\,dm=O(\eps\log\eps^{-1})$.
By~\cite{Baladi07}, this estimate is optimal, so this is a 
situation where linear response fails in contrast to Example~\ref{eg-UE}.
\end{eg}

\section{Families of nonuniformly expanding maps}
\label{sec-ufNUE}

In this section, we consider the situation where the fast dynamics is generated by nonuniformly expanding maps $T_\eps$, such that the nonuniform expansion is uniform in the parameter $\eps$.

In Subsection~\ref{sec-NUE}, we recall the notion of nonuniformly expanding map.
In Subsection~\ref{sec-uf}, we describe the uniformity criteria on the family $T_\eps$ and state our main result on averaging, Theorem~\ref{thm-NUE}, for such families.
In Subsections~\ref{sec-estUE} and~\ref{sec-est}  we establish some basic estimates for uniformly and nonuniformly expanding maps.
In Subsection~\ref{sec-pf} we prove Theorem~\ref{thm-NUE}.

\subsection{Nonuniformly expanding maps}
\label{sec-NUE}

Let $(M,d_M)$ be a locally compact separable bounded metric space with
finite Borel measure $m$ and let $T:M\to M$ be a nonsingular
transformation for which $m$ is ergodic.
Let $Y\subset M$ be a subset of positive measure, and 
let $\alpha$ be an at most countable measurable partition
of $Y$ with $m(a)>0$ for all $a\in\alpha$.    We suppose that there is an $L^1$
{\em return time} function $\tau:Y\to\Z^+$, constant on each $a$ with
value $\tau(a)\ge1$, and constants $\lambda>1$, $\eta\in(0,1]$, $C_0,C_1\ge1$
such that for each $a\in\alpha$,
\begin{itemize}
\item[(1)] $F=T^\tau$ restricts to a (measure-theoretic) bijection from $a$  onto $Y$.
\item[(2)] $d_M(Fx,Fy)\ge \lambda d_M(x,y)$ for all $x,y\in a$.
\item[(3)] $d_M(T^\ell x,T^\ell y)\le C_0d_M(Fx,Fy)$ for all $x,y\in a$,
$0\le \ell <\tau(a)$.
\item[(4)] $\zeta=\frac{dm|_Y}{dm|_Y\circ F}$
satisfies $|\log \zeta(x)-\log \zeta(y)|\le C_1d_M(Fx,Fy)^\eta$ for all
\mbox{$x,y\in a$}.
\end{itemize}
Such a dynamical system $T:M\to M$ is called {\em nonuniformly expanding}.
We refer to $F=T^\tau:Y\to Y$ as the {\em induced map}.
(It is not required that $\tau$ is the first return time to $Y$.)
It follows from standard results (recalled later) that there is a unique absolutely continuous ergodic $T$-invariant probability measure $\nu$ on $M$.

\begin{rmk}  The uniformly expanding maps in Example~\ref{eg-UE} 
are clearly nonuniformly expanding: take $Y=M$, $\eta=1$, $\tau=1$. 
Then conditions (1) and (2) are immediate, condition (3) is vacuously satisfied,
and condition (4) holds with
$C_1=\sup_{x,y\in M,\,x\neq y}|(DT_\eps)_x-(DT_\eps)_y|/d_M(x,y)$.
\end{rmk}

\subsection{Uniformity assumptions}
\label{sec-uf}

Now suppose that $T_\eps:M\to M$, $\eps\in[0,\eps_0)$, is a family of nonuniformly expanding maps
as defined in Subsection~\ref{sec-NUE}, with corresponding absolutely continuous
ergodic invariant probability measures $\nu_\eps$.

\begin{defn}  \label{def-uf}
Let $p>1$.  We say that $T_\eps:M\to M$ is a {\em uniform family} of nonuniformly expanding maps (of order $p$) if 
\begin{itemize}
\item[(i)] The constants $C_0,C_1\ge1$, $\lambda>1$, $\eta\in(0,1]$
can be chosen independent of $\eps\in[0,\eps_0)$.
\item[(ii)] The return time functions
$\tau_\eps:Y_\eps\to\Z^+$ lie in $L^p$ for all $\eps\in[0,\eps_0)$,
	and moreover $\sup_{\eps\in[0,\eps_0)}\int_{Y_\eps}|\tau_\eps|^p\,dm<\infty$.
\end{itemize}
\end{defn}

We can now state our main result for this section.    Recall the set up in
Section~\ref{sec-gen}.

\begin{thm} \label{thm-NUE}
If $T_\eps:M\to M$ is a uniform family of nonuniformly expanding maps of order $p$, then
there is a constant $C>0$ such that
for all $\eps\in[0,\eps_0)$
\[
	\int_M \delta_\eps^{p+d-1}d\nu_\eps\le\begin{cases}
C\eps^{(p-1)/2}, & p>2 \\
	C\eps^{(p-1)^2/p}, & p\in(1,2]  
\end{cases}.
\]
\end{thm}

\begin{rmk} \label{rmk-rate}  In the case $p>2$, it follows from
Theorem~\ref{thm-NUE} that
\[
|\delta_\eps|_{L^q(\nu_\eps)}
=O(\eps^{(p-1)/(2(p+d-1))}),
\]
for all $q\le p+d-1$.  Since $\delta_\eps$ is uniformly bounded,
$|\delta_\eps|_{L^q(\nu_\eps)}
=O(\eps^{(p-1)/(2q)})$ for all $q>p+d-1$.
Similar comments apply for $p\in(1,2]$.

In particular, if $p$ can be taken arbitrarily large in Definition~\ref{def-uf}, then we obtain that
$|\delta_\eps|_{L^q(\nu_\eps)}=O(\eps^{\frac12-})$ for all $q>0$.

If in addition $\nu_0$ is strongly statistically stable with
$R_\eps=\int_M|\rho_\eps-\rho_0|\,dm=O(\eps^{\frac12-})$,
then by Proposition~\ref{prop-gen} we obtain
$|\z|_{L^q(\nu_\eps)}=O(\eps^{\frac12-})$
for all $q>0$ and
$|\z|_{L^1(\nu_0)}=O(\eps^{\frac12-})$.
\end{rmk}

\begin{rmk}  \label{rmk-AV}
Alves \& Viana~\cite{AlvesViana02} prove strong statistical stability for a large class of noninvertible dynamical systems.  These maps are uniform families of nonuniformly expanding maps in a sense that is very similar to our definition.
In fact it is almost the case that their definition includes our definition, so it is almost the case that verifying the assumptions of~\cite{AlvesViana02}
is sufficient to obtain
averaging and rates of averaging via Theorem~\ref{thm-NUE}.

To be more precise, let us momentarily ignore assumption~(3) in Subsection~\ref{sec-NUE}.  Then 
Definition~\ref{def-uf}(i) with $\eta=1$ is immediate from~\cite[(U1)]{AlvesViana02}, and 
Definition~\ref{def-uf}(ii) is immediate from~\cite[(U2$'$)]{AlvesViana02} which follows from their conditions (U1) and (U2).

Hence it remains to discuss assumption~(3).  This assumption is not explicitly mentioned in~\cite{AlvesViana02} since it is not required for the statement of their main results.   However, in specific applications, the hypotheses in~\cite{AlvesViana02} are often verified via the method of~{\em hyperbolic times}~\cite{Alves00}.   When the return time function $\tau_\eps$ is a hyperbolic time, then it is automatic that $C_0=1$ (see for example~\cite[Proposition~3.3(3)]{Alves04}).

Alves {\em et al.}~\cite{AlvesLuzzattoPinheiro05} introduced a general method for constructing inducing schemes, where $\tau_\eps$ is not necessarily a hyperbolic time but is close enough that $C_0$ can still be chosen uniformly.
Alves~\cite{Alves04} combined the methods of~\cite{AlvesLuzzattoPinheiro05} and~\cite{AlvesViana02} to prove statistical stability for large classes of examples.  We show now that 
in the situation discussed by~\cite{Alves04}, assumption (3) holds with uniform $C_0$ and hence our main results hold.
Certain quantities $\delta_1>0$ and $N_0\ge1$ are introduced in~\cite[Lemma~3.2]{Alves04} and~\cite[eq.~(16)]{Alves04} respectively, and are explicitly uniform in $\eps$.  
Moreover $\tau_\eps=n+m$ where
$n$ is a hyperbolic time and $m\le N_0$ 
(see~\cite[Section~4.3]{Alves04}), so $C_0$ depends only on at most $N_0$ iterates of $T_\eps$.  The construction in~\cite{Alves04} (see in particular the proof of~\cite[Lemma~4.2]{Alves04}) ensures that the derivative of $T_\eps$ is bounded along these iterates, so assumption (3) holds and $C_0$ is uniform in $\eps$.

We mention also the extension of~\cite{AlvesLuzzattoPinheiro05} due to
Gou\"ezel~\cite{Gouezel06} where $C_0=1$ (see~\cite[Theorem~3.1 4)]{Gouezel06}.)

Finally, we note that when~\cite{AlvesLuzzattoPinheiro05} is used to obtain polynomial decay of correlations with rate $O(1/n^\beta)$, $\beta>0$, the resulting uniform family is of order $p=\beta+1-$.   (Uniformity in $\eps$ in Definition~\ref{def-uf}(ii) follows from~\cite[Lemma~5.1]{Alves04}.) 
\end{rmk}

\subsection{Explicit estimates for uniformly expanding maps}
\label{sec-estUE}

Throughout this subsection, we work with a fixed uniformly expanding map
$F:Y\to Y$ satisfying conditions (1), (2) and (4).   
Some standard constructions and estimates are described.  The main novelty is that we stress the dependence of various constants on the underlying constants $C_1$, $\lambda$ and $\eta$.
For convenience, we normalise the metric $d_M$ so that $\diam M=1$.

For $\theta\in(0,1)$, we define the
% Fix $\theta\in[\lambda^{-\eta},1)$ and define the
{\em symbolic metric} $d_\theta(x,y)=\theta^{s(x,y)}$ where the
{\em separation time} $s(x,y)$ is the least integer $n\ge0$ such that $F^nx$ and $F^ny$ lie in distinct partition element.  It is assumed that the partition $\alpha$ separates orbits of $F$, so $s(x,y)$ is finite for all $x\neq y$ guaranteeing that $d_\theta$ is a metric.

Given $\phi:Y\to\R^d$, we define $\|\phi\|_\theta=|\phi|_\infty+|\phi|_\theta$ where $|\phi|_\theta=\sup_{x\neq y}|\phi(x)-\phi(y)|/d_\theta(x,y)$.
Then $\phi$ is {\em $d_\theta$-Lipschitz} if $\|\phi\|_\theta<\infty$.
% We say that $\phi:Y\to\R^d$ is {\em locally $d_\theta$-Lipschitz} if 
% $\sup_{y\in a}|\phi(y)|<\infty$ and
% $\sup_{x,y\in a,\,x\neq y}
% |\phi(x)-\phi(y)|/d_\theta(x,y)<\infty$ 
% for each $a\in\alpha$.

The assumptions on $F$ guarantee that there exists a unique absolutely continuous $F$-invariant probability measure $\mu$ on $Y$.
Let $P:L^1(Y)\to L^1(Y)$ denote the (normalised) transfer operator corresponding to $F$ and $\mu$,
so $\int_Y \phi\circ F\,\psi\,d\mu=\int_Y \phi\,P\psi\,d\mu$ for
all $\phi\in L^\infty$ and $\psi\in L^1$.
Define $g:Y\to\R$ by setting $g|_a=d\mu|_a/d(\mu\circ F|_a)$ for $a\in\alpha$.
Then $(P\phi)(y)=\sum_{a\in\alpha}g(y_a)\phi(y_a)$
where $y_a$ is the unique preimage of $y$ under $F$ lying in~$a$.

\begin{lemma} \label{lem-wellknown}
There exist constants $\theta,\gamma\in(0,1)$, $C_2\ge1$ depending continuously on $\lambda$, $\eta$ and $C_1$ such that
\begin{itemize}
\item[(a)] 
$d_M(x,y)^\eta\le d_\theta(x,y)$ for all $x,y\in Y$.
% \begin{prop} \label{prop-mu}
% There is a unique ergodic $F$-invariant probability measure $\mu$ on $Y$ equivalent
% to $m|_Y$.  Moreover, the density $h=d\mu/dm|_Y$ satisfies
% \[
% e^{-C_1(1-\theta)^{-1}}
% \le h\le e^{C_1(1-\theta)^{-1}} \quad\text{and}\quad 
% |h|_{\theta\ell}\le C_1(1-\theta)^{-1}.
% \]
% \end{prop}
\item[(b)] For all $x,y\in a$, $a\in\alpha$,
\begin{align} \label{eq-GM}
g(y)\le C_2\mu(a)\quad\text{and}\quad |g(x)-g(y)|\le C_2\mu(a)d_\theta(x,y).
\end{align}
  \item[(c)] Let $\phi:Y\to\R$ be $d_\theta$-Lipschitz with $\int\phi\,d\mu = 0$.  Then
	\[
|P^n\phi|_\infty \le C_2\gamma^n\|\phi\|_\theta \enspace\text{for all $n\ge1$}.
\]
\end{itemize}
\end{lemma}

\begin{proof}  
Choose $\theta=\lambda^{-\eta}$.  Let $n=s(x,y)$.
By condition~(2),
\[
1\ge \diam Y\ge d_M(F^nx,F^ny)\ge \lambda^nd_M(x,y)=(\theta^{1/\eta})^{-n}d_M(x,y).
\]
Hence $d_M(x,y)^\eta\le \theta^n =d_\theta(x,y)$ proving (a).

By condition~(4) and~\cite[Proposition~2.3]{KKMsub}, there is a constant $K$ depending continuously on $\lambda$, $\eta$ and $C_1$ such that
$|\log g|_\theta\le K$.  Hence for $y\in a$, $a\in\alpha$,
\[
\mu(a)=\int_a g\,d\mu\circ F\ge \inf_a g|_a\, \mu(Fa)= \inf g|_a \ge e^{-K}g(y),
\]
so $g|_a \le e^{K}\mu(a)$.
Next, we note the inequality
$t-1\le t\log t$ which is valid for all $t\ge1$.
Let $x,y\in a$ and suppose without loss that $g(y)\le g(x)$.
Setting $t=g(x)/g(y)\ge1$,
\begin{align*}
\frac{g(x)}{g(y)}-1
& \le \frac{g(x)}{g(y)}\log\frac{g(x)}{g(y)}
\le e^{K} K d_\theta(x,y).
\end{align*}
Hence
$g(x)-g(y)\le g(y) e^{K} K d_\theta(x,y)  
\le e^{2K}K \mu(a) d_\theta(x,y)$.
Hence part~(b) holds with $C_2=e^{2K}K$.

Finally, part~(c) follows for example from~\cite[Proposition~2.5]{KKMsub}.
\end{proof}

\subsection{Explicit estimates for nonuniformly expanding maps}
\label{sec-est}

Throughout this subsection, we work with a fixed nonuniformly expanding map
$T$ satisfying assumptions (1)--(4) and such that $\tau\in L^p$ for some
$p>1$.

There is a standard procedure to pass from the $F$-invariant ergodic absolutely continuous probability measure $\mu$ on $Y$ to
a $T$-invariant ergodic absolutely continuous probability
measure $\nu$ on $M$. We briefly describe this procedure, since the construction
is required in the proof of Lemma~\ref{lem-moment}.
Define the Young tower~\cite{Young99}
\begin{align} \label{eq-tower}
\Delta=\{(y,\ell)\in Y\times\Z:0\le\ell<\tau(y)\},
\end{align}
with probability measure
$\mu_\Delta=\mu\times\{{\rm counting}\}/\int_Y \tau\,d\mu$.  Define $\pi_\Delta:\Delta\to M$, $\pi_\Delta(y,\ell)=T^\ell y$.  Then 
$\nu=(\pi_\Delta)_*\mu_\Delta$ is the desired probability measure on $M$.

In the remainder of this subsection, 
$L^q$ norms of functions defined on $Y$ are computed using $\mu$.
For functions on other spaces, the measures are indicated explicitly in the notation.

Given an observable $v:M\to\R^d$, we define the {\em induced observable} $V:Y\to\R$,
\[
V(y)=\sum_{\ell=0}^{\tau(y)-1}v(T^\ell y).
\]
If $v: M\to\R^d$ satisfies $\int_Mv\,d\nu=0$, then $\int_YV\,d\mu=0$.

\begin{prop}  \label{prop-V}
If $v:M\to\R^d$ is $d_M$-Lipschitz, then
% $V:Y\to\R^d$ is locally $d_\theta$-Lipschitz and 
$PV:Y\to\R^d$ is $d_\theta$-Lipschitz.

Moreover, 
\[
|V(y)|\le \tau(a)|v|_\infty, \quad
|V(x)-V(y)|\le C_0\theta^{-1}\tau(a)\Lip\, v\,d_\theta(x,y),
\;\text{for all $x,y\in a, a\in\alpha$},
\]
and
\[
|PV|_\infty \le C_2|\tau|_1|v|_\infty, \quad 
|PV|_\theta \le C_0C_2\theta^{-1}|\tau|_1\|v\|_\Lip.
\]
\end{prop}

\begin{proof}  
The estimate for $V(y)$ is immediate.
By condition~(3) and Lemma~\ref{lem-wellknown}(a),
\begin{align*}
& |V(x)-V(y)|  \le \Lip\, v\,\sum_{\ell=0}^{\tau(a)-1}d_M(T^\ell x,T^\ell y)
\le C_0\Lip\, v\,\sum_{\ell=0}^{\tau(a)-1}d_M(F x,F y)
\\ & \qquad  = C_0\tau(a)\Lip\, v\,d_M(F x,F y)
 \le C_0\tau(a)\Lip\, v\, d_\theta(F x,F y)^{1/\eta}
 \\ & \qquad  \le C_0\tau(a)\Lip\, v\,d_\theta(F x,F y)
 = C_0\theta^{-1}\tau(a)\Lip\, v\,d_\theta(x,y),
\end{align*}
completing the estimates for $V$.

Next, $(PV)(y)=\sum_{a\in\alpha}g(y_a)V(y_a)$, so
by~\eqref{eq-GM},
\[
|PV|_\infty \le C_2\sum_{a\in\alpha}\mu(a)\tau(a)|v|_\infty=C_2|\tau|_1|v|_\infty.
\]
Also,
\begin{align*}
& |(PV)(x)-(PV)(y)|  \le \sum_{a\in\alpha}|g(x_a)-g(y_a)||V(x_a)|+
\sum_{a\in\alpha}g(y_a)|V(x_a)-V(y_a)| \\
& \qquad  \le C_2\sum_{a\in\alpha}\mu(a)d_\theta(x,y)\tau(a)|v|_\infty+
C_2\sum_{a\in\alpha}\mu(a)C_0\theta^{-1}\tau(a)\Lip\,v\,d_\theta(x,y),
\end{align*}
yielding the estimate for $|PV|_\theta$.
\end{proof}

\begin{prop} \label{prop-mchi}
Let $p\ge1$.
There exists $m\in L^p(Y,\R^d)$ and $\chi\in L^\infty(Y,\R^d)$ such that
$V=m+\chi\circ F-\chi$, and $m\in \ker P$.
Moreover,
\begin{align*}
|m|_p & \le 
3C_3|\tau|_p\|v\|_\Lip, \quad \text{and} \quad
|\chi|_\infty  \le C_3|\tau|_1\|v\|_\Lip,
\end{align*}
where $C_3=2C_0C_22\theta^{-1}(1-\gamma)^{-1}$.
\end{prop}

\begin{proof}
By Proposition~\ref{prop-V} and Lemma~\ref{lem-wellknown}(c) with
$\phi=PV$, for $n\ge1$,
\begin{align*}
	|P^nV|_\infty & \le C_2\gamma^{n-1}\|PV\|_\theta
	\le 2C_0C_2^2\theta^{-1}\gamma^{n-1}|\tau|_1\|v\|_\Lip.
\end{align*}
It follows that $\chi=\sum_{k=1}^\infty P^kV$ lies in $L^\infty$ and
$|\chi|_\infty\le C_3|\tau|_1\|v\|_\Lip$.

Write $V=m+\chi\circ F-\chi$; then $m\in L^p$ and $Pm=0$.
Finally, $|m|_p\le |V|_p+2|\chi|_\infty\le |\tau|_p|v|_\infty+2|\chi|_\infty\le 3C_3|\tau|_p\|v\|_\Lip$.
\end{proof}

\begin{cor} \label{cor-Burk}
Define $V_n=\sum_{j=0}^{n-1}V\circ F^j$.
Let $p>1$.
There exists a constant $C_4\ge 1$ depending only on $p$ and $C_3$ such that
\[
\bigl|\max_{1\le j\le n}|V_j|\,\bigr|_p \le 
	C_4|\tau|_p\|v\|_\Lip\ n^{\max\{1/2,1/p\}}. 
\]
\end{cor}

\begin{proof}
First note that $V_n=m_n+\chi\circ F^n-\chi$ where
$m_n=\sum_{j=0}^{n-1}m\circ F$.
Since $m\in\ker P$, an application of Burkholder's inequality~\cite{Burkholder73} shows that 
\[
	\bigl|\max_{1\le j\le n}|m_j|\bigr|_p\le C(p)|m|_p n^{\max\{1/2,1/p\}},
\]
(see for example the proof of~\cite[Proposition~4.3]{MZ15}).
Hence 
\[
\bigl|\max_{1\le j\le n}|V_j|\,\bigr|_p
\le C(p)|m|_p n^{\max\{1/2,1/p\}}+ 2|\chi|_\infty.
\]
The result follows from Proposition~\ref{prop-mchi}
with $C_4=5C_3C(p)$.
\end{proof}

\begin{lemma} \label{lem-moment}  
Let $p>1$.  Let $v_n=\sum_{j=0}^{n-1}v\circ T^j$.  Then
\[
	\bigl|\max_{j\le n}|v_j|\bigr|_{L^{p-1}(\nu)}
	\le 
	5C_4|\tau|_p^{p/(p-1)}\|v\|_\Lip\,n^{\max\{1/2,1/p\}}.
\]
\end{lemma}

\begin{proof}  
Let $q=p-1$.
Define the tower $\Delta$ as in~\eqref{eq-tower} with
tower map $f:\Delta\to\Delta$  where 
\[
f(y,\ell)=\begin{cases} (y,\ell+1), & \ell\le \tau(y)-2
\\ (Fy,0), & \ell=\tau(y)-1 \end{cases}.
\]
Recall that $\mu_\Delta=\mu\times{\rm counting}/\bar \tau$ on $\Delta$
where $\bar\tau=\int_Y \tau\,d\mu$.
Also, $\nu=(\pi_\Delta)_*\mu_\Delta$
where $\pi_\Delta:\Delta\to M$ is the projection
$\pi_\Delta(y,\ell)=T^\ell y$.

Let $\hat v=v\circ \pi_\Delta$ and
define $\hat v_n=\sum_{j=0}^{n-1}\hat v\circ f^j$.  
Then $\int_M |v_n|^q\,d\nu=\int_\Delta |\hat v_n|^q\,d\mu_\Delta$.

	Next, let $N_n:\Delta\to\{0,1,\dots,n\}$ be the number of laps by time $n$,
	\[
		N_n(y,\ell)=\#\{j\in\{1,\dots,n\}:f^j(y,\ell)\in Y\times\{0\}\}.
\]
Then
\[
\hat v_n(y,\ell)=V_{N_n(y,\ell)}(y)+H\circ f^n(y,\ell)-H(y,\ell),
\]
where $H(y,\ell)=\hat v_\ell(y,0)$.
Note that 
$|H(y,\ell)|\le |v|_\infty\tau(y)$ for all $(y,\ell)\in\Delta$.

Now $f^n(y,\ell)=(F^{N_n(y,\ell)}y,\ell+n-\tau_{N_n(y,\ell)}(y))$, so
\begin{align*}
	 \max_{j\le n}|H\circ f^j & (y,\ell)|
	  \le |v|_\infty\max_{j\le n}\tau(F^{N_j(y,\ell)}y)
	  \le |v|_\infty\max_{j\le n}\tau(F^jy)
\\ & \le |v|_\infty \tau(y)
+ |v|_\infty \max_{1\le j\le n}\tau(F^jy)=
|v|_\infty \hat\tau(y,\ell)
+ |v|_\infty \max_{1\le j\le n}\hat\tau(F^jy,\ell),
\end{align*}
where $\hat\tau:\Delta\to\Z^+$ is given by
$\hat\tau(y,\ell)=\tau(y)$.

We estimate the first term in $L^q(\mu_\Delta)$ and the second term in 
$L^p(\mu_\Delta)$.  Using the definition of $\mu_\Delta$ and the fact that
$\bar\tau\ge1$, 
\begin{align*}
	\int_\Delta \hat\tau^q\,d\mu_\Delta =(1/\bar\tau)\int_Y\tau^{q+1}\,d\mu
	\le \int_Y\tau^p\,d\mu,
\end{align*}
so $|\hat\tau|_{L^q(\mu_\Delta)}\le |\tau|_p^{p/(p-1)}$.
Also,
\begin{align*}
	\int_\Delta\max_{1\le j\le n}\hat\tau(F^jy)^p & \,d\mu_\Delta
 =(1/\bar\tau) \int_Y\tau\max_{1\le j\le n}\tau^p\circ F^j\,d\mu
\le(1/\bar\tau)\sum_{j=1}^n \int_Y\tau\,\tau^p\circ F^j\,d\mu
\\ & =(1/\bar\tau)\sum_{j=1}^n \int_Y P\tau\,\tau^p\circ F^{j-1}\,d\mu
\le (1/\bar\tau)\sum_{j=1}^n |P\tau|_\infty \int_Y\tau^p\circ F^{j-1}\,d\mu
\\ & = (1/\bar\tau)  n |P\tau|_\infty \int_Y\tau^p\,d\mu
\le  C_2n \int_Y\tau^p\,d\mu,
\end{align*}
where we used Proposition~\ref{prop-V} with $v=1$ (and hence $V=\tau$) for the final inequality.
Hence
\[
	\bigl|\max_{1\le j\le n}\hat\tau(F^jy)\bigr|_{L^q(\mu_\Delta)} \le
	\bigl|\max_{1\le j\le n}\hat\tau(F^jy)\bigr|_{L^p(\mu_\Delta)} \le C_2^{1/p}n^{1/p}|\tau|_p.
\]
Combining these estimates, we obtain that
\[
	|H|_{L^q(\mu_\Delta)}\le 
	\bigl|\max_{j\le n}|H\circ f^j|\bigr|_{L^q(\mu_\Delta)}\le 
	2C_2^{1/p}|v|_\infty|\tau|_p^{p/(p-1)}n^{1/p}.
\]

  Next, using H\"older's inequality,
  \begin{align*}
	  \int_\Delta \max_{j\le n} |V_{N_j(y,\ell)}(y)|^q\,d\mu_\Delta
& \le \int_\Delta \max_{j\le n}|V_j(y)|^q\,d\mu_\Delta
=(1/\bar\tau) \int_Y \tau \max_{j\le n}|V_j|^q\,d\mu
\\ & \le   \ |\tau|_p \ \bigl|\max_{j\le n}|V_j|^q\bigr|_{p/q}
 =  \ |\tau|_p \ \bigl|\max_{j\le n}|V_j|\bigr|_p^q.
\end{align*}
By Corollary~\ref{cor-Burk},
\[
\Bigl|\max_{j\le n}|V_{N_j}|\Bigr|_{L^q(\mu_\Delta)}
\le C_4 |\tau|_p^{p/(p-1)}\|v\|_\Lip\ n^{\max\{1/2,1/p\}}.
\]

By the triangle inequality, using that $C_2^{1/p}\le C_4$,
\begin{align*}
	\bigl||\max_{j\le n}v_n|\bigr|_{L^{p-1}(\nu)}
& =\bigl|\max_{j\le n}|\hat v_n|\bigr|_{L^q(\mu_\Delta)}\le 
 5C_4|\tau|_p^{p/(p-1)}\|v\|_\Lip\  n^{\max\{1/2,1/p\}},
\end{align*}
as required.
\end{proof}

\subsection{Proof of Theorem~\ref{thm-NUE}}
\label{sec-pf}

Define $v_{\eps,x}$ and $v_{\eps,x,n}$ as in Section~\ref{sec-gen}.
Note that $\Lip\,v_{\eps,x}\le 2L$ for all $\eps,x$ and $\int_M v_{\eps,x}\,d\nu_\eps=0$.

It follows from Lemma~\ref{lem-moment} that
\[
	\bigl|\max_{j\le n}|v_{\eps,x,j}|\bigr|_{L^{p-1}(\nu_\eps)}
	\le 10C_4L|\tau_\eps|_p^{p/{(p-1)}} n^{\max\{1/2,1/p\}},
\]
for all $\eps\ge0$, $x\in\R$, $n\ge1$.
By~\eqref{eq-delta},
\begin{align*}
	\int_M \delta_{1,\eps}^{p+d-1} \, d\nu_\eps & \le
	\Gamma \eps^{p-1} \sup_{x\in E} \int_M \max_{n\le 1/\eps}|v_{\eps,x,n}|^{p-1} \, d\nu_\eps 
	 \le \Gamma C_4^{p-1}|\tau_\eps|_p^p \eps^{p-1} \eps^{-(p-1)\max\{1/2,1/p\}}
	\\ & = \Gamma C_4^{p-1}|\tau_\eps|_p^p \eps^{p-1} \eps^{\min\{-(p-1)/2,-(p-1)/p\}}
	  = \Gamma C_4^{p-1}|\tau_\eps|_p^p  \eps^{\min\{(p-1)/2,(p-1)^2/p\}}.
\end{align*}
We obtain the same estimate for $\delta_{2,\eps}$ replacing $v_{\eps,x}$, $v_{\eps,x,n}$ by
$V_{\eps,x}$, $V_{\eps,x,n}$.

\section{Examples: Nonuniformly expanding maps}
\label{sec-egNUE}

\begin{eg}[Intermittent maps] \label{eg-LSV}
Let $M=[0,1]$ with Lebesgue measure $m$ and consider the intermittent maps $T:M\to M$ given by
\begin{align} \label{eq-LSV}
Tx=\begin{cases}
x(1+2^a x^a), & x\in[0,\frac12] \\ 2x-1, & x\in(\frac12,1] \end{cases}.
\end{align}
These were studied in~\cite{LiveraniSaussolVaienti99}.
Here $a>0$ is a parameter.   For $a\in(0,1)$ there
is a unique absolutely continuous invariant probability measure with
$C^\infty$ nonvanishing density.

We consider a family $T_\eps:M\to M$, $\eps \in [0,\eps_0)$, of such intermittent maps with parameter
$a_\eps\in(0,1)$
depending continuously on $\eps$.   Let $\nu_\eps$ denote 
the corresponding family of absolutely continuous invariant probability measures.

For each $\eps$, we take $Y=[\frac12,1]$ and
let $\tau_\eps:Y\to\Z^+$ be the first return time $\tau_\eps(y)=\inf\{n\ge1:T_\eps^ny\in Y\}$.
Define the first return map $F_\eps=T_\eps^{\tau_\eps}:Y\to Y$.
Let $\alpha_\eps=\{Y_\eps(n),\,n\ge1\}$
where $Y_\eps(n)=\{y\in Y:\tau_\eps(y)=n\}$.

It is standard that each $T_\eps$ is a nonuniformly expanding map in the sense of Section~\ref{sec-NUE} with $\tau_\eps\in L^p$ for all $p<1/a_\eps$.
Fix $p\in(1,1/a_0)$ and choose $0<a_-<a_0<a_+<1$ such that $p<1/a_+$.
  Without loss we can shrink $\eps_0$ so that $a_\eps\in[a_-,a_+]$
  for all $\eps\in[0,\eps_0)$.
	  We show that $T_\eps$, $\eps\in[0,\eps_0)$, satisfies Definition~\ref{def-uf} for this choice of $p$.

  Since $T'_\eps\ge1$ on $M$ and $T'_\eps=2$ on $Y$, it is immediate that conditions 
  (1)---(3) in Section~\ref{sec-NUE} are satisfied with $\lambda=2$ and $C_0=1$.
  
  Recall that $\zeta_{\eps} = \frac{dm|_Y}{dm|_Y \circ F_\eps}$. Note that 
  $\zeta_{\eps}(y) = 1/F_\eps'(y)$.
  For each $Y_\eps(n) \in \alpha_\eps$, define 
  the bijection $F_{\eps,n}=F_\eps|_{Y_\eps(n)}: Y_\eps(n) \to M$.
  Let $G_{\eps,n} = (F_{\eps,n}^{-1})'
  = \zeta_\eps \circ F_{\eps,n}^{-1}$.
  By \cite[Assumption~A2 and Theorem~3.1]{Korepanov15}, there is a constant $K$,
  depending only on $a_-$ and $a_+$, such that 
  $|(\log G_{\eps,n})'| \leq K$ for all $\eps\in[0,\eps_0)$.
  Hence
	  $|(\log  \zeta_{\eps}\circ F_{\eps,n}^{-1})'|
	  = |(\log G_{\eps,n})'| \le K$.
  By the mean value theorem, for $x,y \in Y_\eps(n)$,
  \[
    |\log \zeta_{\eps}(x) - \log \zeta_{\eps}(y)| 
    =
    |(\log  \zeta_{\eps}\circ F_{\eps,n}^{-1})(F_\eps x) - (\log \zeta_{\eps}\circ F_{\eps,n}^{-1})(F_\eps y)| 
    \leq 
    K  | F_{\eps} x - F_{\eps} y|.
  \]
  This proves condition (4) in Section~\ref{sec-NUE}, and so
  condition~(i) in Definition~\ref{def-uf} is satisfied.
  
  Define $x_1=\frac12$ and inductively $x_{n+1}<x_n$ (depending on $\eps$)
  with $T_\eps x_{n+1}=x_n$.
  Then $T_\eps(Y_\eps(n))=[x_n,x_{n-1}]$ for $n\ge2$ and it is standard that
  $x_n=O(1/n^\alpha)$ as a function of $n$.  By~\cite[Lemma~5.2]{Korepanov15},
  there is a constant $K$,
  depending only on $a_-$ and $a_+$, such that
  $x_n\le Kn^{-1/\alpha_+}$ for all $n\ge1$, $\eps\in[0,\eps_0)$.
	  Hence
	  \[
m(\tau_\eps>n)=m([1/2,(x_n+1)/2])=x_n/2\le Kn^{-1/\alpha_+}.
\]
Since $p<1/\alpha_+$ it follows that $\sup_{\eps\in[0,\eps_0)}\int_Y |\tau_\eps|^p \, dm<\infty$ so
 condition~(ii) in Definition~\ref{def-uf} is satisfied.

By~\cite{BaladiTodd15,Korepanov15}, $\nu_\eps$
is strongly statistically stable and the densities $\rho_\eps$ satisfy
$R_\eps=\int|\rho_\eps-\rho_0|\,dm=O(a_\eps - a_0)$.
Hence, by Corollary~\ref{cor-gen}, we obtain averaging in $L^1$ with respect to $\nu_\eps$ and also with respect to any absolutely continuous probability measure.

Finally, if $\eps\mapsto a_\eps$ is Lipschitz say, so that $R_\eps=O(\eps)$, then we obtain the rates described in 
Remark~\ref{rmk-rate} with $p=(1/a_0)-$.
\end{eg}

\begin{eg}[Logistic family] \label{eg-logistic}
We consider the family of quadratic maps $T:[-1,1]\to[-1,1]$ given by 
$T(x)=1-ax^2$, $a\in[-2,2]$, with $m$ taken to be Lebesgue measure.  

Let $b,c>0$.  The map $T$ satisfies the Collet-Eckmann condition~\cite{ColletEckmann83} with constants $b,c$ if
\begin{align} \label{eq-CE}
|(T^n)'(1)|\ge ce^{bn}\quad\text{ for all $n\ge0$}.
\end{align}
In this case, we write $a\in Q_{b,c}$.
The set of Collet-Eckmann parameters is given by $P_1=\bigcup_{b,c>0}Q_{b,c}$
and is a Cantor set of
positive Lebesgue measure~\cite{Jakobson81,BenedicksCarleson85}.
When $a\in P_1$, the map $T$
has an invariant set $\Lambda$ consisting of a finite union of intervals with an ergodic absolutely continuous invariant probability measure $\nu_a$.
The density for $\nu_a$ is bounded below on $\Lambda$ and lies in $L^{2-}$.
The invariant set attracts Lebesgue almost every trajectory in $[-1,1]$.

There is also an open dense set of parameters $P_0\subset[-2,2]$ for which
$T$ possesses a periodic sink attracting 
Lebesgue almost every trajectory in $[-1,1]$.
By Lyubich~\cite{Lyubich02}, $P_0\cup P_1$ has full measure. 
For $a\in P_0$, we let $\nu_a$ denote the invariant probability measure supported on the
periodic attractor, so we have a map $a\mapsto \nu_a$ defined on $P_0\cup P_1$.

It is clear that statistical stability
holds on $P_0$,
and that strong statistical stability fails everywhere in $P_0\cup P_1$.
Moreover, Thunberg~\cite[Corollary~1]{Thunberg01} shows that on any full measure subset of $E\subset [-2,2]$ the map $a\to\nu_a$ is not statistically stable
at any point of $P_1\cap E$.
On the other hand, 
Freitas \& Todd~\cite{FreitasTodd09} proved that strong statistical stability holds on $Q_{b,c}$ for all constants $b,c>0$.  That is, the map
$a\to\rho_a=d\nu_a/dm$ from $Q_{b,c}\to L^1$ is continuous.
(See also~\cite{Freitas05,Freitas10} for the same result restricted to the Benedicks-Carleson parameters~\cite{BenedicksCarleson85}.)

We consider families $\eps\to T_\eps$ where each $T_\eps$ is a quadratic map with
parameter $a=a_\eps$ depending continuously on $\eps$.
Fix $b,c>0$ such that $a_0\in Q_{b,c}$.
We claim that
\begin{align*}
\lim_{\stackrel{\eps\to0}{a_\eps\in Q_{b,c}}}\int\z\,d\nu_{a_\eps}=0.
\end{align*}
Moreover, using Corollary~\ref{cor-dm}  we obtain convergence in $L^1(\nu)$ for every absolutely continuous probability measure $\nu$.
Given the above results on strong statistical stability, it suffices to verify that $T_\eps$ is a uniform family of nonuniformly expanding maps.

For the Benedicks-Carleson parameters, the method in~\cite{Freitas05,Freitas10} is the approach of~\cite{Alves04} and we can apply Remark~\ref{rmk-AV}.  
In the general case, a different method exploiting negative Schwarzian derivative and Koebe spaces~\cite[Proof of Theorem~B in Section~6]{FreitasTodd09} shows that the conditions in~\cite{AlvesViana02} are satisfied.  
By Remark~\ref{rmk-AV}, this completes the proof of averaging with the possible exception of condition (3).  
However, a standard consequence of negative Schwarzian derivative and the Koebe distortion property (as discussed in~\cite[Lemma~4.1]{FreitasTodd09} and used
in~\cite[Remark~3.2]{FreitasTodd09}) is that bounded distortion holds at intermediate steps and not just at the inducing time as in condition~(4).  Hence there is a uniform constant $\tilde C_1$ such that 
$\frac{|T_\eps^jx-T_\eps^jy|}{\diam T_\eps^ja} \le \tilde C_1 \frac{|F_\eps x-F_\eps y|}{\diam Y_\eps}$ for all partition elements $a$, all $x,y\in a$ and all $j\le\tau_\eps(a)$.
In particular, 
$|T_\eps^jx-T_\eps^jy| \le (2\tilde C_1/\diam Y_\eps) |F_\eps x-F_\eps y|$, yielding condition~(3) uniformly in $\eps$.

Next, we discuss rates of convergence.
By~\cite[Lemma~4.1]{FreitasTodd09}, condition~(ii) in Definition~\ref{def-uf} is satisfied for any $p>1$.
Hence
$|\delta_\eps|_{L^q(\nu_\eps)}=O(\eps^{\frac12-})$.   If $\eps\mapsto a_\eps$ is $C^1$, then it follows from
Baladi {\em et al.}~\cite{BaladiBenedicksSchnellmann15} that $R\eps=\int|\rho_\eps-\rho_0|\,dm=O(\eps^{\frac12-})$.
By Remark~\ref{rmk-rate}, we obtain averaging with rate $O(\eps^{\frac12-})$ in $L^q(\nu_\eps)$ for all $q>0$ and in $L^1(\nu_0)$.
\end{eg}

\begin{eg}[Multimodal maps]
Freitas \& Todd~\cite{FreitasTodd09} also consider families of multimodal maps
where each critical point $c$ satisfies a Collet-Eckmann condition along the orbit of $Tc$ with constants uniform in $\eps$.
Hence the averaging result for the quadratic family in Example~\ref{eg-logistic} extends immediately to multimodal maps.
\end{eg}

\begin{eg}[Viana maps] \label{eg-viana}

Viana~\cite{Viana97} introduced a $C^3$ open class of multi-dimensional nonuniformly expanding maps $T_\eps:M\to M$.  For definiteness, we restrict attention to the case $M=S^1\times\R$.   
Let $S:M\to M$ be the map
$S(\theta,y)=(16\theta\bmod1,a_0+a\sin 2\pi\theta-y^2)$.
Here $a_0$ is chosen so that $0$ is a preperiodic point for the quadratic map $y\mapsto a_0-y^2$ and $a$ is fixed sufficiently small.
Let
$T_\eps$, $0\le\eps<\eps_0$ be a continuous family of $C^3$ maps sufficiently close to $S$.
It follows from~\cite{Alves00,AlvesViana02} that there is an interval $I\subset(-2,2)$ such that, for each $\eps\in[0,\eps_0)$, there is a unique absolutely continuous $T_\eps$-invariant ergodic probability measure $\nu_\eps$ supported in the interior of $S^1\times I$.  Moreover the invariant set $\Lambda_\eps=\supp\nu_\eps$ attracts almost every initial condition in $S^1\times I$.

By Alves \& Viana~\cite{AlvesViana02},  $\nu_0$ is strongly statistically stable.   Moreover, the inducing method of~\cite{AlvesLuzzattoPinheiro05}  and the arguments in~\cite{Alves04} apply to this example, so $T_\eps$ is a uniform family of nonuniformly expanding maps by Remark~\ref{rmk-AV}.
Also, Corollary~\ref{cor-dm} is applicable by Remark~\ref{rmk-dm}.
Hence we obtain averaging in $L^1(\nu_\eps)$ and in $L^1(\mu)$ for all absolutely continuous $\mu$.
\end{eg}

\section{Averaging for continuous time fast-slow systems}
\label{sec-flow}

Let $\phi_t^\eps:M\to M$, $0\le \eps<\eps_0$, be a family of semiflows
defined on the metric space $(M,d_M)$.
For each $\eps\ge0$, let $\nu_\eps$ denote
a $\phi_t^\eps$-invariant ergodic Borel probability measure.
Let $a:\R^d \times M\times[0,\eps_0)\to \R^d$
be a family of vector fields on $\R^d$ satisfying
conditions~\eqref{eq-L1}---\eqref{eq-L3}.

We consider the family of fast-slow systems
  \begin{align*}
& \dotx   = \eps a(\x,\y,\eps), \quad \x(0)=x_0,   \\
& \y(t)   = \phi_t^\eps y_0, 
  \end{align*}
where the initial condition $\x(0)=x_0$ is fixed throughout.
The initial condition $y_0\in M$ is again chosen randomly with respect to various measures that are specified in the statements of the results.

  Define $\hatx:[0,1]\to\R^d$ by setting
$\hatx(t)=\x(t/\eps)$.
Let $X:[0,1]\to\R^d$ be the solution to the ODE~\eqref{eq-ODE}
and define
\[
	\z = \sup_{t\in[0,1]}|\hatx(t)-X(t)|.
\]

Recall that $E=\{x\in\R^d:|x-x_0|\le L_1\}$.
As in Section~\ref{sec-order},
  define $\bar a(x,\eps)=\int_M a(x,y,\eps)\,d\nu_\eps(y)$ and let
$v_{\eps,x}(y)=a(x,y,\eps)-\bar a(x,\eps)$.
We define the \emph{order function} $\delta_\eps=\delta_{1,\eps}+\delta_{2,\eps}:M\to\R$ where
\begin{align*}
	\delta_{1,\eps} & = \sup_{x\in E} \sup_{0 \leq t \leq 1/\eps} 
  \eps|v_{\eps, x, t}| \qquad \text{ where } \qquad v_{\eps, x, t}=
  \int_0^t v_{\eps,x}\circ \phi_s^\eps\,ds, \\
	\delta_{2,\eps} & = \sup_{x\in E} \sup_{0 \leq t \leq 1/\eps} 
  \eps|V_{\eps, x, t}| \qquad \text{ where } \qquad V_{\eps, x, t}=
  \int_0^t (Dv_{\eps,x})\circ \phi_s^\eps\,ds.
\end{align*}

The next results is the continuous time analogue of Theorem~\ref{thm-gen}.
The proof is entirely analogous, and hence is omitted.

\begin{thm} \label{thm-genflow}  
Let $S_\eps  = \sup_{x\in E}|\int_M a(x,y,0)\,(d\nu_\eps-d\nu_0)(y)|+\eps$.
 Assume conditions~\eqref{eq-L1}---\eqref{eq-L3}.
If $\delta_\eps\le\frac12$,
then $\z  \le 6e^{2L}(\delta_\eps+S_\eps)$.
\end{thm}

As in the discrete time setting, we say that $\nu_0$ is statistically stable if
$\nu_\eps\to_w\nu_0$.  Proposition~\ref{prop-ss} goes through unchanged and statistical stability implies that $S_\eps\to0$.

If the measures $\nu_\eps$ are absolutely continuous with respect to a reference measure $m$ on $M$, we define the densities $\rho_\eps=d\nu_\eps/dm$ and set
$R_\eps=\int_M|\rho_\eps-\rho_0|\,dm$.   Then $\nu_0$ is strongly statistically stable if $R_\eps\to0$.  Proposition~\ref{prop-gen} and
Corollaries~\ref{cor-gen} and~\ref{cor-dm} go through unchanged from
the discrete time setting.

Fix a Borel subset $M'\subset M$ and a reference Borel measure
$m'$ on $M'$.
Let $h_\eps:M'\to\R^+$ be a family of Lipschitz functions such that $\phi^\eps_{h_\eps(y)}(y)\in M'$ for almost all $y\in M'$.  The map $T_\eps:M'\to M'$, $T_\eps(y)=\phi_{h_\eps(y)}^\eps(y)$, is then defined almost everywhere.  

As usual, we suppose that there is a family $\nu'_\eps$ of
ergodic $T_\eps$-invariant probability measures on $M'$.  
Define the suspension $M_{h_\eps}=\{(y,u)\in M'\times\R:0\le u\le h_\eps(y)\}/\sim$
	where $(y,h_\eps(y))\sim (T_\eps y,0)$.
	The suspension semiflow $f_t^\eps:M_{h_\eps}\to M_{h_\eps}$  is given by
$f_t^\eps(y,u)=(y,u+t)$ computed modulo identifications.
Let $\bar h_\eps=\int_{M'} h_\eps\,d\nu'_\eps$.
Then $\nu''_\eps=(\nu'_\eps\times{\rm Lebesgue})/\bar h_\eps$ is an ergodic absolutely continuous
$f_t^\eps$-invariant probability measure on $M_{h_\eps}$.
The projection $\pi_\eps:M_{h_\eps}\to M$ given by
$\pi_\eps(y,u)=\phi_u^\eps y$ is a semiconjugacy between $f_t^\eps$ and $\phi_t^\eps$.
Hence $\nu_\eps=\pi_{\eps\,*}\,\nu''_\eps$ is an ergodic $\phi_t^\eps$-invariant probability measure on $M$.

We suppose from now on that there are constants $K_2\ge K_1\ge1$ such that
for all $x,y\in M'$, $\eps\in[0,\eps_0)$,
\begin{itemize}
	\item $K_1^{-1}\le h_\eps\le K_1$, $\Lip\, h_\eps\le K_1$, $|h_\eps-h_0|_\infty \le K_1\eps$.
\item $d_M(\phi_t^\eps x,\phi_t^\eps y)\le K_2 d_M(x,y)$  and 
$d_M(\phi_t^\eps y,\phi_t^0 y)\le K_2 \eps$ for all $t\le K_1$.
\end{itemize}
(These assumptions are easily weakened; in particular changing the $\eps$ estimates to $\eps^{1/2}$ will not affect anything.)

\begin{prop} \label{prop-ssflow1}
Let $v:M\to\R^d$ be Lipschitz.
Define $\tilde v:M'\to\R^d$,
$\tilde v(y)=\int_0^{h_0(y)}v(\phi^0_u y)\,du$.  Then
\[
\int_M v(d\nu_\eps-d\nu_0)\le 3K_2^4\|v\|_\Lip \eps+
K_1^3|v|_\infty\Bigl|\int_{M'}h_0(d\nu_\eps'-d\nu_0')\Bigr|
+K_1\Bigl|\int_{M'}\tilde v(d\nu_\eps'-d\nu_0')\Bigr|.
\]
\end{prop}

\begin{proof}
We have
\begin{align*}
	\int_M v\,(d\nu_\eps-d\nu_0) & =
	\int_{M_{h_\eps}} v\circ\pi_\eps\,d\nu''_\eps-\int_{M_{h_0}} v\circ\pi_0\,d\nu''_0  \\ & =
	(1/\bar h_\eps)\int_{M'}\int_0^{h_\eps} v\circ\pi_\eps\,du\,d\nu'_\eps
	-(1/\bar h_0)\int_{M'}\int_0^{h_0} v\circ\pi_0\,du\,d\nu'_0
	\\ & =I_1+I_2+I_3+I_4
\end{align*}
where
\begin{align*}
I_1 & = (1/\bar h_\eps-1/\bar h_0)
	\int_{M'}\int_0^{h_\eps} v\circ\pi_\eps\,du\,d\nu'_\eps,
\; 
I_2  = (1/\bar h_0)
\int_{M'}\int_0^{h_\eps} (v\circ\pi_\eps-v\circ\pi_0)\,du\,d\nu'_\eps,
\\
I_3 & = (1/\bar h_0)
	\Bigl(\int_{M'}\int_0^{h_\eps} v\circ\pi_0\,du\,d\nu'_\eps-
	\int_{M'}\int_0^{h_0} v\circ\pi_0\,du\,d\nu'_\eps\Bigr),
\\
I_4 & = (1/\bar h_0)
	\Bigl(\int_{M'}\int_0^{h_0} v\circ\pi_0\,du\,d\nu'_\eps-
	\int_{M'}\int_0^{h_0} v\circ\pi_0\,du\,d\nu'_0\Bigr).
\end{align*}
Now
\begin{align*}
|I_1| & \le K_1^2|\bar h_\eps-\bar h_0|K_1|v|_\infty
	\le K_1^3|v|_\infty \Bigl(K_1\eps+\Bigl|\int_{M'}h_0\,(d\nu_\eps'-d\nu_0')\Bigr|\Bigr), 
\\
|I_2| & \le K_1^2\sup_{y\in M'}\sup_{0\le u\le K_1}\Lip\, v\, d_M(\phi_u^\eps y,\phi_u^0 y)
\le K_1^2K_2\Lip\, v\,  \eps,
\\
|I_3| & \le K_1|v|_\infty|h_\eps-h_0|_\infty \le K_1^2|v|_\infty  \eps,
\qquad
|I_4|  \le K_1\Bigl|\int_{M'}\tilde v\,(d\nu_\eps'-d\nu_0')\Bigr|.
\end{align*}
The result follows from the combination of these estimates.
\end{proof}

\begin{cor} \label{cor-ssflow2}  Statistical stability of $\nu'_0$ implies
statistical stability of $\nu_0$.
\end{cor}

\begin{proof}  
This follows from Proposition~\ref{prop-ssflow1}.
\end{proof}

\begin{cor} \label{cor-ssflow}
	Suppose that the measures $\nu_\eps'$ are absolutely continuous with respect to $m'$, with densities $\rho_\eps'=d\nu_\eps'/dm'$.  Then
$\BIG S_\eps\le 3K_2^4L\Bigl(\eps+\int_{M'}|\rho_\eps'-\rho_0'|\,dm'\Bigr)$.
\end{cor}

\begin{proof}
This follows from Proposition~\ref{prop-ssflow1} with $v(y)=a(x,y,0)$ for each fixed $x$.
\end{proof}

Next we show how the order function for the flows $\phi_t^\eps:M\to M$ is related to the order function for the maps $T_\eps:M'\to M'$.
We restrict attention to $\delta_{1,\eps}$ since the corresponding statement for $\delta_{2,\eps}$ is identical.

Define the family of induced observables $w_{x,\eps}:M'\to\R$,
\[
	w_{\eps,x}(y)=\int_0^{h_\eps(y)}v_{\eps,x}(\phi_u^\eps y)\,du.
\]
Note that $\int_{M'}w_{\eps,x}\,d\nu'_\eps=0$ and
$w_{\eps,x}$ is $d_M$-Lipschitz with $\|w_\eps\|_\Lip\le 2K_1K_2L$.
Let
\[
\Delta_{1,\eps}=\sup_{x\in E}\sup_{1\le n\le 1+K_1/\eps} \eps|w_{\eps,x,n}|
\qquad\text{where}\qquad
w_{\eps,x,n}=\sum_{j=0}^{n-1}w_{\eps,x}\circ T_\eps^j.
\]

We can now state our main result for this section.

\begin{lemma} \label{lem-vw}
	Let $q\ge1$.  Then 
	$|\delta_{1,\eps}|_{L^q(\nu_\eps)} \le (|\Delta_{1,\eps}|_{L^q(\nu'_\eps)}+4K_1L\eps)$.
\end{lemma}

\begin{proof}
	Let $\hat v_{\eps,x}=v_{\eps,x}\circ \pi_\eps$ and
	define $\hat v_{\eps,x,t}=\int_0^t\hat v_{\eps,x}\circ f_u^\eps\,du$.
Let $N_{\eps,t}:M_{h_\eps}\to\{0,1,\dots,1+[K_1t]\}$ be the number of laps by time $t$,
\[
N_{\eps,t}(y,u)=\#\{s\in(0,t]:f_s^\eps(y,u)\in M'\times\{0\}\}.
\]
Then
\[
\hat v_{\eps,x,t}(y,u)=w_{\eps,x,N_{\eps,t}(y,u)}(y)+H_{\eps,x}\circ f_t^\eps(y,u)-H_{\eps,x}(y,u),
\]
where $H_{\eps,x}(y,u)=\int_0^u \hat v_{\eps,x}(y,u')\,du'=\hat v_{\eps,x,u}(y,0)$.
Note that $|H_{\eps,x}|_\infty\le 2K_1L$.
Hence
  \begin{align*}
	  \sup_{s\le t}|v_{\eps,x,s}|\circ\pi_\eps(y,u) & =
	  \sup_{s\le t}|\hat v_{\eps,x,s}(y,u)|
	  \le \sup_{s\le t}|w_{\eps,x,N_s(y,u)}(y)|+4K_1L
	  \\ & \le \max_{j\le 1+K_1t}|w_{\eps,x,j}(y)|+4K_1L.
  \end{align*}
It follows that
\[
\eps \sup_{s\le 1/\eps}|v_{\eps,x,s}|\circ\pi_\eps(y,u)\le\Delta_{1,\eps}(y)+4 K_1L\eps,
\]
  and so 
  $\delta_{1,\eps}\circ\pi_\eps(y,u)\le \Delta_{1,\eps}(y)+4K_1L\eps$.
  The result follows.
\end{proof}

As a consequence of Proposition~\ref{cor-ssflow} and Lemma~\ref{lem-vw}, our results for maps go through immediately for semiflows.  For example, suppose that the maps $T_\eps(y)=\phi_{h_\eps(y)}^\eps(y)$ are a family of quadratic maps as in Example~\ref{eg-logistic}.  Then for any $q>0$, we obtain averaging in $L^q(\nu_\eps)$ 
with rate $O(\eps^{\frac12-})$.
If moreover, $\nu_0$ is strongly statistically stable, then we obtain averaging
in $L^1(\nu_0)$ with rate $O(\eps^{\frac12-})$, and averaging
in $L^1(\mu)$ for any absolutely continuous probability measure $\mu$ on $M$.

\section{Counterexample for almost sure convergence}
\label{sec-ae}

It is known~\cite{BakhtinKifer08} that almost sure convergence fails for fully-coupled fast-slow systems.  Here we give an example to show that almost sure convergence fails also in the simpler context of families of skew products as considered in this paper.

Fix $\beta>0$.  We consider the family of maps $T_\eps:[0,1]\to[0,1]$ given by
$T_\eps y=2y+\eps^\beta \bmod1$ with invariant measure $\nu_\eps$ taken to be
Lebesgue for all $\eps\ge0$.
Let $a(x,y,\eps)=\cos 2\pi y$.  Since $a$ has mean zero, the averaged ODE
is given by $\dot X=0$.  We take $x_0=0$ so that
$X(t)\equiv0$.  Nevertheless, we prove:
\begin{prop} \label{prop-ae}
For every $y_0\in[0,1]$, 
  $\limsup_{\eps \to 0} \hatx(1)=1$.
\end{prop}

\begin{proof}
Let $y_0\in [0,1]$ and 
$\delta>0$.
Let $N= {[\delta^{-1/2}]}$ and choose an
integer $1 \leq k \leq 2^N$ such that
\[
  y_0 \in [ -\delta^\beta + (k-1)2^{-N}, \,\, 
   -\delta^\beta + k 2^{-N}  ].
\]
Choose $\eps$ such that
\[
  y_0 = -\eps^\beta + (k-1) 2^{-N}.
\]
Then
\[
  \delta^\beta - 2^{-N} \leq \eps^\beta \leq \delta^\beta.
\]
If $\delta$ is small enough, then
$\delta^\beta - 2^{-N} = \delta^\beta - 2^{-{[\delta^{-1/2}]}}>0$,
and so $0<\eps \leq \delta$.

Now
\[
\y_n=2^ny_0+(2^n-1)\eps^\beta \bmod1
=-\eps^\beta +(k-1)2^{n-N} \bmod1,
\]
for all $n\ge0$.
In particular, for $n\ge N$ we have
$\y_n = -\eps^\beta\bmod1$, and $\cos 2\pi \y_n \geq 1-\pi\eps^\beta$.
Note that $N \leq \eps^{-1/2}$.
Hence
$\hatx(1)
 =  \eps \sum_{n=0}^{[\eps^{-1}]-1} \cos(2\pi \y_n)
 = 1+O(\eps^{1/2})+O(\eps^\beta)$.
 Since $\eps\in(0,\delta]$ is arbitrarily small, the result follows.
\end{proof}

\appendix

\section{Proof of second order averaging}
\label{sec-second}

In this appendix, we prove Theorem~\ref{thm-gen}.
This is a quantitative version of a result due to~\cite{Sanders83}
with a somewhat simplified proof.
We work with discrete time rather than continuous time.

First, we consider the case where $T:M\to M$ is independent of $\eps$.  Suppose that
$a:\R^d\times M\to\R^d$ and $\bar a:\R^d\to\R^d$ are functions.
Assume that $\|a\|_\Lip\le L$ and $\|Da\|_\Lip\le L$ where
$D=\frac{d}{dx}$ and $L\ge1$.

For $\eps>0$, consider the discrete fast-slow system
\begin{align*}
    x_{n+1} & = x_n + \eps a(x_n,y_n), \quad
    y_{n+1}  = Ty_n
\end{align*}
with $x_0\in\R^d$ and $y_0\in M$ given.

Define $\hat x_\eps:[0,1]\to\R^d$,  $\hat x_\eps(t) = x_{[t/\eps]}$, and let $X:[0,1]\to\R^d$ be
the solution of the ODE
    $\dot X = \bar a(X)$ with initial condition $X(0) = x_0$.
    As in Section~\ref{sec-gen}, we define
$\delta_\eps=\delta_{1,\eps}+\delta_{2,\eps}$ where
\begin{align*}
	\delta_{1,\eps}(y_0) & =\sup_{x}\sup_{1\leq n\leq 1/\epsilon}\eps\Bigl|\sum_{j=0}^{n-1}(a(x,y_j)-\bar{a}(x))\Bigr|, \\
	\delta_{2,\eps}(y_0) & =\sup_{x}\sup_{1\leq n\leq 1/\epsilon}\epsilon\Bigl|\sum_{j=0}^{n-1}(Da(x,y_j)-D\bar{a}(x))\Bigr|.
\end{align*}

\begin{thm} \label{thm-second} 
Let $\eps>0$, $x_0\in\R^d$.  For all $y_0\in M$ with $\delta_\eps(y_0)\le\frac12$ and $t\in[0,1]$,
\[
	|\hat{x}_{\eps}(t)-X(t)|\leq 5e^{2L}(\delta_\eps(y_0)+\eps).
\]
\end{thm}

First we recall a discrete version of Gronwall's lemma.

\begin{prop} \label{prop-discreteG}
Suppose that $b_n\ge0$ and that there exist constants $C,D\ge0$ such that
$b_n\le C+D\sum_{m=0}^{n-1}b_m$ for all $n\ge0$.
Then $b_n\le C(D+1)^n$.
\end{prop}

\begin{proof} This follows by induction.
\end{proof}

Define a function $u:\mathbb{R}^{d}\times\{ 0,2,\ldots,1/\eps\} $
by setting $u(x,0)\equiv0$ and
\[
u(x,n)=\frac{\eps}{\delta_\eps}\sum_{j=0}^{n-1}(a(x,y_j)-\bar{a}(x)), \quad n\ge1.
\]

\begin{prop}
\label{prop-u}For any $n\le 1/\eps$, we have $|u(\cdot, n)|_\infty\le1$ and
$\Lip\,u(\cdot,n)\leq1$. \end{prop}

\begin{proof}
For all $x$,
\[
|Du(x,n)|=\frac{\eps}{\delta_\eps}\Bigl|\sum_{j=0}^{n-1}(Da(x,y_j)-D\bar{a}(x))\Bigr|\leq\frac{\delta_{2,\eps}}{\delta_\eps}\leq1.
\]
Hence the second estimate follows from the mean value theorem, and the first estimate is easier.
\end{proof}

Define a new sequence $w_n$ by setting $w_0=x_0$ and 
\[
w_n=x_n-\delta_\eps u(w_{n-1},n), \quad n\ge1.
\]
 
\begin{lemma} \label{lem-w}
For all $0\leq n\leq 1/\eps$ 
and for all $y\in M$ with $\delta_\eps(y)\le\frac12$,
\[
\Bigl|w_n-w_0-\eps\sum_{k=0}^{n-1}\bar{a}(w_k)\Bigr|\le 4L\delta_\eps.
\]
\end{lemma}

\begin{proof}
By definition, for all $0\leq k\leq 1/\eps$, 
\[
\delta_\eps[u(w_k,k+1)-u(w_k,k)]=\eps[a(w_k,y_k)-\bar{a}(w_k)].
\]
Thus 
\begin{align*}
& \delta_\eps u(w_{n-1},n) = \delta_\eps\sum_{k=0}^{n-1}\{ u(w_k,k+1)-u(w_{k-1},k)\} \\
 & \qquad = \delta_\eps\Bigl(\sum_{k=0}^{n-1}\{ u(w_k,k+1)-u(w_k,k)\} )
 +\sum_{k=0}^{n-1}\{ u(w_k,k)-u(w_{k-1},k)\} \Bigr)\\
 & \qquad = \eps\sum_{k=0}^{n-1}\{a(w_k,y_k)-\bar{a}(w_k)\}+{\rm I}_n,
\end{align*}
where 
\[
{\rm I}_n=\delta_\eps\sum_{k=1}^{n-1}\{ u(w_k,k)-u(w_{k-1},k)\} .
\]
This together with the definition of $x_n$ yields
\begin{align*}
w_n &= x_0+ \eps\sum_{k=0}^{n-1}a(x_k,y_k)-\delta_\eps u(w_{n-1},y,n)
\\ & = w_0+ \eps\sum_{k=0}^{n-1}a(x_k,y_k)-\eps\sum_{k=0}^{n-1}\{a(w_k,y_k)-\bar{a}(w_k)\}-{\rm I}_n
\\ &  = w_0+\eps\sum_{k=0}^{n-1}\bar{a}(w_k)-{\rm I_n}+{\rm II}_n,
\end{align*}
where
\[
{\rm II}_n=\eps\sum_{k=0}^{n-1}\{ a(x_k,y_k)-a(w_k,y_k)\}.
\]
We claim that for all $0\leq n\leq 1/\eps$, 
\begin{align*} 
|w_{n+1}-w_n-\eps\bar{a}(w_n)|\le 4L\eps\delta_\eps.
\end{align*}
The result follows by summing over $n$.

It remains to prove the claim.   It is easy to check that $w_1-w_0-\eps\bar{a}(w_0)=0$.
Inductively, suppose that
$|w_n-w_{n-1}-\eps\bar{a}(w_{n-1})|\le 4L\eps\delta_\eps$.
Notice that 
\[
|{\rm II}_{n+1}-{\rm II}_n|\leq\eps\Lip a|x_n-w_n|\leq\Lip a\,\eps\delta_\eps|u(w_{n-1},n)|\leq L\eps\delta_\eps.
\]
Also, 
\begin{align*}
|{\rm I}_{n+1}-{\rm I}_n| & \le \delta_\eps\Lip u|w_n-w_{n-1}|
  \le \delta_\eps(\eps|\bar{a}(w_{n-1})|+4L\eps\delta_\eps)
  \le 3L\eps\delta_\eps,
\end{align*}
where the second inequality follows by the induction hypothesis and the third inequality uses $\delta_\eps\le\frac12$. Therefore 
\[
|w_{n+1}-w_n-\eps\bar{a}(w_n)|\leq 
|{\rm I}_{n+1}-{\rm I}_n|+|{\rm II}_{n+1}-{\rm II}_n|
\le 4L\eps\delta_\eps,
\]
proving the claim.
\end{proof}

Define the sequence 
\[
z_{n+1}=z_n+\eps\bar{a}(z_n), \quad z_0=x_0.
\]

\begin{cor} \label{cor-xz}
$|x_n-z_n|\leq 5\delta_\eps e^{2L}$ for all $1\leq n\leq 1/\eps$,
and for all $y\in M$ with $\delta_\eps(y)\le\frac12$,
\end{cor}

\begin{proof}
Write
\begin{align*}
w_n-z_n=w_n-x_0-\eps\sum_{k=0}^{n-1}\bar a(z_k)
=w_n-w_0-
\eps\sum_{k=0}^{n-1}\bar a(w_k)
+\eps\sum_{k=0}^{n-1}\{\bar a(w_k)-\bar a(z_k)\}.
\end{align*}
By Lemma \ref{lem-w}, 
\begin{align*}
|w_n-z_n| & \le 4L\delta_\eps+\eps\sum_{k=0}^{n-1}|\bar{a}(w_k)-\bar{a}(z_k)|
\le 4L\delta_\eps+L\eps \sum_{k=0}^{n-1}|w_k-z_k|.
\end{align*}
By Proposition~\ref{prop-discreteG},
$|w_n-z_n|\le   4L\delta_\eps e^L\le 4\delta_\eps e^{2L}$.
Moreover,
$|x_n-w_n|\leq\delta_\eps|u|_{\infty}\leq\delta_\eps$
and the result follows.
\end{proof}

\begin{lemma}\label{lm:ga5}
	$\BIG |X(n\eps)-z_n | \leq L^2e^L\eps$.
\end{lemma}

\begin{proof}
Write
\[
	X(n\eps)=x_0+\sum_{m=0}^{n-1}\int_{m\eps}^{(m+1)\eps}[\bar a(X(s))-\bar a(X(m\eps))]\,ds+\eps\sum_{m=0}^{n-1}\bar a(X(m\eps)).
\]
Since $|\bar a(X(t_1))-\bar a(X(t_2))|\le L|X(t_1)-X(t_2)|\le  L^2|t_1-t_2|$ for all $t_1,t_2$, we obtain that
\[
	\Bigl|X(n\eps)-x_0-\eps\sum_{m=0}^{n-1}\bar a(X(m\eps))\Bigr|\le L^2n\eps^2\le L^2\eps.
\]
Hence 
\[
|X(n\eps)-z_n|\le \eps\sum_{m=0}^{n-1}|\bar a(X(m\eps))-\bar a(z_m)|+L^2\eps
\le L\eps\sum_{m=0}^{n-1}|X(m\eps)-z_m|+L^2\eps.
\]
The result follows from Proposition~\ref{prop-discreteG}.
\end{proof}

\begin{pfof}{Theorem~\ref{thm-second}}
  By  Corollary~\ref{cor-xz} and Lemma~\ref{lm:ga5},
  \begin{align*}
	  |x_{[t/\eps]} - X(t)| & \leq |X(t) - X([t/\eps ] \eps)|
    + | X([t/\eps ] \eps) - z_{[t/\eps ]}| 
    + | z_{[ t/\eps ]} - x_{[t/\eps]}| \\
    & \leq \eps L+L^2e^L\eps+ 5e^{2L} \delta_\eps
       \leq 5e^{2L}( \delta_\eps+\eps).
  \end{align*}
  This completes the proof.
\end{pfof}

\begin{pfof}{Theorem~\ref{thm-gen}}
Replacing $T$, $a(x,y)$ and $\bar a(x)$ in 
Theorem~\ref{thm-second} by
$T_\eps$, $a(x,y,\eps)$ and $\bar a(x,\eps)$, we obtain that
\[
    |\hat x_\eps(t) - X_\eps(t)| \le  5e^{2L}( \delta_\eps+\eps)
  \]
where
\[
    \dot X_\eps = \bar a(X,\eps), \quad X_\eps(0) = x_0.
\]

Let $A_\eps=\sup_{x\in E}|\bar a(x,\eps)-\bar a(x,0)|$.  Then
\begin{align*}
|X_\eps(t)-X(t)| & \le \int_0^t|\bar a(X_\eps(s),\eps)-\bar a(X(s),0)|\,ds
\\ & \le tA_\eps+\int_0^t|\bar a(X_\eps(s),0)-\bar a(X(s),0)|\,ds
\\ & \le A_\eps+L\int_0^t|X_\eps(s)-X(s)|\,ds.
\end{align*}
By Gronwall's lemma,
$
|X_\eps(t)-X(t)|  \le 
 e^LA_\eps$ for all $t\le 1$.

Next, $A_\eps\le L\eps+\sup_{x\in E}|\int_M a(x,y,0)\,(d\nu_\eps-d\nu_0)(y)|$.
Combining these estimates we obtain that
\[
|\hat x_\eps(t) - X(t)| \le 5e^{2L} \delta_\eps+6e^{2L}\eps
+e^L\sup_{x\in E}\Bigl|\int_M a(x,y,0)\,(d\nu_\eps-d\nu_0)(y)\Bigr|,
\]
yielding the result.
\end{pfof}

\paragraph{Acknowledgements}
This research was supported in part by a
European Advanced Grant {\em StochExtHomog} (ERC AdG 320977).
We are grateful to Vitor Ara\'ujo, Jorge Freitas, Vilton Pinheiro, Mike Todd and Paulo Varandas for helpful discussions.

%%%%%%%%%%%%%%%%%

\end{document}